\documentclass[11pt]{article}

\usepackage[margin=1in]{geometry}
\usepackage[svgnames]{xcolor}
\usepackage{amsmath, amssymb, amsthm, mathtools}
\usepackage{enumitem}
\usepackage{hyperref}
\usepackage[nameinlink, capitalize]{cleveref}
\usepackage{tikz}
\usetikzlibrary{arrows.meta, positioning}

\hypersetup{
  colorlinks = true,
  linkcolor = blue!50!black,
  citecolor = blue!50!black,
  urlcolor = blue!50!black,
  pdftitle = {Formalization of the Galerkin Construction for the Two-Dimensional Navier--Stokes Equations in Lean},
  pdfauthor = {Weinan Wang},
  pdfkeywords = {Lean formalization, Navier--Stokes equations, Galerkin method, Leray--Hopf solutions, spectral compactness}
}

\newtheorem{theorem}{Theorem}[section]
\newtheorem{proposition}[theorem]{Proposition}

\theoremstyle{definition}
\newtheorem{remark}[theorem]{Remark}

\newcommand{\R}{\mathbb{R}}
\newcommand{\ip}[2]{\left\langle #1, #2 \right\rangle}
\newcommand{\norm}[1]{\left\lVert #1 \right\rVert}
\newcommand{\Bil}{\mathcal{B}}

\title{Formalization of the Galerkin Construction\\
for the Two-Dimensional Navier--Stokes Equations in Lean}
\author{Weinan Wang}
\date{}

\begin{document}

\maketitle

\begin{abstract}
We formalize in Lean~4 the Galerkin construction of global Leray--Hopf weak
solutions for the unforced two-dimensional Navier--Stokes equations on
arbitrary bounded open domains with no-slip boundary conditions.  For every
initial datum in the solenoidal \(L^2\) velocity space, the solution has a
weakly continuous velocity path, satisfies the weak equation locally in time, and obeys
the energy inequality at every time.  For rectangles, we also formalize
finite-horizon solutions with continuous forcing in the dual energy space.
The development
includes divergence-free graph spaces, compact spectral coordinates,
transport cancellation, Ladyzhenskaya estimates, and simultaneous
space--time compactness.  The abstract Hilbert-space results apply to
evolution equations with a compact energy-to-velocity embedding and a skew
trilinear nonlinearity.
\end{abstract}


\medskip


\section{Introduction}

The foundational works of Leray \cite{Leray1934} and Hopf \cite{Hopf1950}
established the modern weak theory of the incompressible Navier--Stokes
equations.  Their constructions combine approximation with kinetic-energy
estimates that survive weak convergence, producing the
global finite-energy solutions now called Leray--Hopf solutions.  In two space dimensions, the
Ladyzhenskaya interpolation inequality supplies the additional integrability
needed to control the quadratic transport term in the energy class.  This
combination of Galerkin approximation, compactness, and energy estimates is a
standard existence method for nonlinear parabolic equations; see
\cite{Ladyzhenskaya1969,Temam1977}.  Detailed treatments of the
Navier--Stokes theory appear in the monographs of Constantin and Foias,
Robinson, Rodrigo, and Sadowski, and Bedrossian and Vicol
\cite{ConstantinFoias1988,RobinsonRodrigoSadowski2016,BedrossianVicol2022}.

Let \(\Omega\subset\mathbb R^2\) be an arbitrary bounded open set.  Consider
the normalized, unforced incompressible Navier--Stokes system on \(\Omega\):
\begin{equation}\label{eq:navier-stokes-problem}
  \partial_tu-\Delta u+(u\cdot\nabla)u+\nabla p=0,
  \qquad \nabla\cdot u=0,
  \qquad u|_{\partial\Omega}=0,
  \qquad u(0)=u_0.
\end{equation}
For arbitrary \(\Omega\), the boundary condition is imposed by requiring the
energy-class velocity to lie in the \(H^1\)-closure of compactly supported
divergence-free fields.  For initial data in
the corresponding divergence-free velocity space, the Leray--Hopf theory
seeks a velocity with finite kinetic energy and square-integrable gradient,
solving \cref{eq:navier-stokes-problem} against divergence-free tests and
satisfying
\[
  \norm{u(t)}_{L^2(\Omega)}^2
  +2\int_0^t\norm{\nabla u(s)}_{L^2(\Omega)}^2\,ds
  \leq \norm{u_0}_{L^2(\Omega)}^2
\]
at every time.  The classical Galerkin proof constructs finite-dimensional
solutions, derives this estimate before taking a limit, obtains strong
compactness in the velocity space, and then passes the quadratic transport term;
the standard compactness step is often obtained from the
Aubin--Lions--Simon lemma~\cite{Aubin1963,Simon1986}.  A spectral compactness
argument instead combines uniform control of the high modes, obtained from
compactness of the spatial embedding, with equicontinuity of every fixed
finite projection, obtained from the equation.  Its ingredients are
finite-dimensional Arzel\`a--Ascoli compactness and operator-norm approximation
after a compact map.

The velocity and velocity--gradient norms give distinct completions.
Convection takes values in the dual of the energy space, while
finite-dimensional norm equivalence yields bounds that depend on the
Galerkin dimension.  The space--time limit has a representative attaining
the initial value and satisfying the energy inequality at every time.  A
single subsequence supports these limiting statements.

The bounded-domain proof uses an auxiliary rectangle \(Q\) with
\(\overline\Omega\subset Q^\circ\).  Compactly supported \(C^1\),
divergence-free fields with zero face values generate a velocity--gradient
graph space \(V_Q\) and a velocity space \(H_Q\).  The coordinate projection
\(J_Q:V_Q\to H_Q\) is proved compact by a Fourier Rellich argument on the
completed graph.  Applying the compact self-adjoint spectral theorem to
\(J_Q^*J_Q\) constructs paired Hilbert bases of \(V_Q\) and \(H_Q\), positive
singular values, and an increasing sequence of finite index sets whose union
is the full index set.  The resulting finite-dimensional subspaces of \(H_Q\)
admit exact lifts to \(V_Q\) and compatible test projections.  For each
Galerkin index, Riesz representation produces
the coefficient ODE, transport skew-symmetry gives the exact energy identity,
and the two-dimensional Ladyzhenskaya estimate gives a uniform
\(L^2(0,T;V_Q')\) bound on the time derivatives.

The compactness argument combines finite-mode equicontinuity with
operator-norm convergence of the finite-rank approximations.  It yields strong
\(L^2(0,T;H_Q)\) convergence and, along the same subsequence, weak
\(L^2(0,T;V_Q)\) convergence.  Ladyzhenskaya interpolation then gives the
strong \(L^2(0,T;L^4(Q))\) convergence needed for convection.  A diagonal
subsequence defines a weakly continuous representative, and weak lower
semicontinuity in the product of the velocity and gradient spaces yields the energy
inequality at every time.

Compactly supported smooth solenoidal fields on \(\Omega\) define its
velocity and graph energy spaces.  Zero extension to \(Q\) transfers the
compact embedding and the transport and interpolation estimates to these
spaces.  The Galerkin
solutions agree on overlapping time intervals by uniqueness of their
finite-dimensional coefficient equations.  A single diagonal subsequence
then yields one weakly continuous velocity path on \([0,\infty)\), with
energy-class representatives and the weak equation on every finite horizon.

Lean~4 and mathlib provide the proof-assistant
setting~\cite{MouraUllrich2021,Mathlib2020}.  Related work in formalized
analysis includes the Lean formalization of the local \(h\)-principle for
open, ample first-order partial differential relations
~\cite{VanDoornMassotNash2023}, the Gagliardo--Nirenberg--Sobolev inequality
~\cite{VanDoornMacbeth2024}, Schwartz functions, tempered distributions, and
Fourier-defined Sobolev spaces~\cite{Doll2025}, and a generalized Carleson
theorem~\cite{VanDoornCarleson2026}.  Sun's PDE library develops introductory
heat-equation theory in Lean~\cite{SunPDE2026}, while the Lax--Milgram theorem
has been formalized in Coq~\cite{BoldoEtAl2017}.

Formalized PDE theorems in Lean now include the core interior
De Giorgi--Nash--Moser theory for elliptic equations
by Armstrong and Kempe~\cite{ArmstrongKempe2026}, which they describe as,
to their knowledge, the first machine-checked formalization of a major
theorem in modern PDE theory.  Armstrong and Kuusi
formalized the coarse-graining arguments underlying quantitative stochastic
homogenization
~\cite{ArmstrongKuusi2025,ArmstrongKuusiCoarseGraining2026}.
Miller formalized Dobrushin's mean-field derivation of the nonlinear Vlasov
equation, including well-posedness, stability, and the mean-field
limit~\cite{Miller2026}.
Bedrossian formalized nonlinear Landau damping for the Vlasov--Poisson
equations in Gevrey regularity~\cite{Bedrossian2026}.  OpenAI has released
Lean formalizations of forced finite-time breakdown for three-dimensional
Navier--Stokes on \(\mathbb R^3\) and \(\mathbb R^3/\mathbb Z^3\), and of an
unforced finite-time singularity for the Euler equation on
\(\mathbb R^3\)~\cite{OpenAINavierStokesEuler2026}.  The present theorem
establishes weak existence from arbitrary energy-class data on bounded
no-slip planar domains.  Its proof combines finite-dimensional dynamics,
space--time compactness, a nonlinear limit, weak temporal continuity, and an
inherited energy inequality.

\begin{remark}[Independent contemporaneous formalization]
Uda's contemporaneous Lean formalization~\cite{UdaLerayHopf2026} and the
present formalization were produced independently.  Uda proves global-in-time
existence for the unforced three-dimensional problem on the periodic torus
and on \(\mathbb R^3\), for every positive viscosity.
\end{remark}

\section{Main Results}\label{sec:main-result}

Let \(\Omega\subset\mathbb R^2\) be bounded and open, and choose an auxiliary
rectangle \(Q\) with \(\overline\Omega\subset Q^\circ\).  Let \(\mathcal D_\sigma(\Omega)\)
be the compactly supported \(C^1\) divergence-free fields in \(\Omega\),
extended by zero to \(Q\).  Define \(H_\Omega\) as their closure in
\(L^2(Q;\mathbb R^2)\) and \(V_\Omega\) as the closure of their
velocity--gradient graphs in
\(L^2(Q;\mathbb R^2)\times L^2(Q;\mathbb R^{2\times2})\).
Write \(J_\Omega:V_\Omega\to H_\Omega\) and
\(G_\Omega:V_\Omega\to L^2(Q;\mathbb R^{2\times2})\) for the
coordinate maps, and \(b_\Omega\) for the continuous convection form.
Elements of these closed spaces vanish almost everywhere on
\(Q\setminus\Omega\).

\begin{theorem}[Global solution on a bounded open domain]
\label{thm:open-domain-global}
For every \(u_0\in H_\Omega\), there is a path
\(u_c:[0,\infty)\to H_\Omega\) and, for each integer \(N\geq0\), a function
\(u_V^N\in L^2([0,N];V_\Omega)\) such that
\(u_c(0)=u_0\) and \(t\mapsto(u_c(t),y)_{H_\Omega}\) is continuous for
every \(y\in H_\Omega\).  On \([0,N]\),
\(J_\Omega u_V^N=u_c\) almost everywhere, and for every \(t\in[0,N]\),
\[
  \norm{u_c(t)}_{H_\Omega}^2
  +2\int_0^t\norm{G_\Omega u_V^N(s)}_{L^2(Q)}^2\,ds
  \leq \norm{u_0}_{H_\Omega}^2.
\]
For every \(\phi\in V_\Omega\) and scalar time test
\(\eta\in C^1([0,N])\) with \(\eta(N)=0\),
\begin{align*}
  &-\int_0^N (u_c(t),J_\Omega\phi)_{H_\Omega}\eta'(t)\,dt
  +\int_0^N (G_\Omega u_V^N(t),G_\Omega\phi)_{L^2(Q)}\eta(t)\,dt\\
  &\qquad
  +\int_0^N b_\Omega(u_V^N(t),u_V^N(t),\phi)\eta(t)\,dt
  =(u_0,J_\Omega\phi)_{H_\Omega}\eta(0).
\end{align*}
\end{theorem}

The next two statements give the finite-interval result on \(Q\) used in the
proof and its forced counterpart.  Write
\(Q=[a_1,b_1]\times[a_2,b_2]\), with \(a_i<b_i\).
Let \(\mathcal D_\sigma(Q)\) consist of pairs \((u,Du)\) such that
\(u\in C_c^1(\mathbb R^2;\mathbb R^2)\),
\(\operatorname{supp}u\subseteq\overline Q\), \(Du\) agrees
with the Fr\'echet derivative of \(u\) on \(Q^\circ\), \(u\) is zero on every
face, \(\operatorname{tr}Du=0\) on \(Q\), and \(u,Du\in L^2(Q)\).  Define
\begin{align*}
  H_Q
  &:={\overline{\operatorname{span}
      \{u:u\in\mathcal D_\sigma(Q)\}}}^{L^2(Q;\mathbb R^2)},\\
  V_Q
  &:={\overline{\operatorname{span}
      \{(u,Du):u\in\mathcal D_\sigma(Q)\}}}
      ^{L^2(Q;\mathbb R^2)\times L^2(Q;\mathbb R^{2\times2})}.
\end{align*}
First-coordinate projection gives \(J_Q:V_Q\to H_Q\), and second-coordinate
projection gives the closed gradient
\(G_Q:V_Q\to L^2(Q;\mathbb R^{2\times2})\).  The graph norm satisfies
\[
  \norm{v}_{V_Q}^2
  =\norm{J_Qv}_{H_Q}^2+\norm{G_Qv}_{L^2(Q)}^2.
\]
The map \(J_Q\) is compact, injective, and has dense range.  The same dense
class determines a bounded skew convection form
\(b_Q:V_Q^3\to\mathbb R\) representing
\(\int_Q((u\cdot\nabla)v)\cdot w\) on smooth fields.

For \(I=[a,b]\), set
\[
  R_I(u_0)=\sqrt{1+\lvert I\rvert}\,\norm{u_0}_{H_Q}.
\]

\begin{theorem}[Unforced two-dimensional Leray--Hopf solution]
\label{thm:box-energy-weak-limit}
Let \(a\leq b\) and \(u_0\in H_Q\).  There are
\[
  u_H\in L^2(I;H_Q),
  \qquad
  u_V\in L^2(I;V_Q),
  \qquad
  u_c:I\to H_Q
\]
such that
\[
  J_Q u_V=u_H\quad\text{in }L^2(I;H_Q),
  \qquad
  u_c(t)=u_H(t)\quad\text{for almost every }t\in I,
\]
\[
  u_H\in L^\infty(I;H_Q),
  \qquad
  \norm{u_H}_{L^\infty(I;H_Q)}\leq\norm{u_0}_{H_Q},
  \qquad
  \norm{u_V}_{L^2(I;V_Q)}\leq R_I(u_0).
\]
For every \(y\in H_Q\), the scalar map
\(t\mapsto (u_c(t),y)_{H_Q}\) is continuous.  Moreover,
\[
  u_c(a)=u_0,
  \qquad
  \norm{u_c(t)}_{H_Q}\leq\norm{u_0}_{H_Q}
  \quad(t\in I),
\]
and, for every \(t\in I\),
\[
  \norm{u_c(t)}_{H_Q}^2
  +2\int_a^t
    \norm{G_Qu_V(s)}_{L^2(Q;\mathbb R^{2\times2})}^2\,ds
  \leq \norm{u_0}_{H_Q}^2.
\]
In particular, for every \(\phi\in V_Q\) and every scalar time test
\(\eta\in C^1(I)\) satisfying \(\eta(b)=0\),
\begin{align*}
  &-\int_I (u_H(t),J_Q\phi)_{H_Q}\eta'(t)\,dt
   +\int_I (G_Qu_V(t),G_Q\phi)_{L^2(Q)}\eta(t)\,dt\\
  &\qquad
   +\int_I b_Q(u_V(t),u_V(t),\phi)\eta(t)\,dt
   =(u_0,J_Q\phi)_{H_Q}\eta(a).
\end{align*}
\end{theorem}

For the forced statement, fix \(C_Q\geq0\) such that
\begin{equation}\label{eq:main-box-poincare}
  \norm{J_Qv}_{H_Q}\leq C_Q\norm{G_Qv}_{L^2(Q)}
  \qquad(v\in V_Q),
\end{equation}
whose existence is established in \cref{prop:box-poincare}.  A continuous map
\(\mathcal F:\mathbb R\to V_Q'\) is called an \emph{admissible forcing} if
there are \(r\in C(\mathbb R;V_Q)\) and a continuous nonnegative function
\(f:\mathbb R\to\mathbb R\), called a \emph{control function}, such that
\[
  \mathcal F(t)(w)=(r(t),w)_{V_Q},
  \qquad
  \lvert\mathcal F(t)(w)\rvert\leq f(t)\norm{w}_{V_Q}
  \qquad(t\in\mathbb R,\ w\in V_Q).
\]
Every \(r\in C(\mathbb R;V_Q)\) defines an admissible forcing with control
function \(f(t)=\norm{r(t)}_{V_Q}\).

\begin{theorem}[Forced two-dimensional Leray--Hopf solution]
\label{thm:forced-leray-hopf}
Let \(a\leq b\), let \(u_0\in H_Q\), and let \(\mathcal F\) be an admissible
forcing with control function \(f\).  Put
\[
  R=\Bigl(\norm{u_0}_{H_Q}^2
    +(1+C_Q^2)\int_If(t)^2\,dt\Bigr)^{1/2}.
\]
Then there are
\[
  u_H\in L^2(I;H_Q),
  \qquad
  u_V\in L^2(I;V_Q),
  \qquad
  u_c:I\to H_Q
\]
such that
\[
  J_Qu_V=u_H,
  \qquad
  u_c(t)=u_H(t)\quad\text{for almost every }t\in I,
\]
\[
  u_H\in L^\infty(I;H_Q),
  \qquad
  \norm{u_H}_{L^\infty(I;H_Q)}\leq R,
  \qquad
  \norm{G_Qu_V}_{L^2(I;L^2(Q))}\leq R,
\]
and \(\norm{u_V}_{L^2(I;V_Q)}\leq\sqrt{1+\lvert I\rvert}\,R\).  For every
\(y\in H_Q\), the scalar map \(t\mapsto(u_c(t),y)_{H_Q}\) is continuous, and
\[
  u_c(a)=u_0,
  \qquad
  \norm{u_c(t)}_{H_Q}\leq R
  \quad(t\in I).
\]
For every \(t\in I\),
\begin{equation}\label{eq:forced-energy-inequality}
  \norm{u_c(t)}_{H_Q}^2
  +2\int_a^t\norm{G_Qu_V(s)}_{L^2(Q)}^2\,ds
  \leq\norm{u_0}_{H_Q}^2
    +2\int_a^t\mathcal F(s)\bigl(u_V(s)\bigr)\,ds,
\end{equation}
and for every \(\phi\in V_Q\) and every scalar time test
\(\eta\in C^1(I)\) satisfying \(\eta(b)=0\),
\begin{equation}\label{eq:forced-weak-equation}
\begin{aligned}
  &-\int_I (u_H(t),J_Q\phi)_{H_Q}\eta'(t)\,dt
   +\int_I (G_Qu_V(t),G_Q\phi)_{L^2(Q)}\eta(t)\,dt\\
  &\qquad
   +\int_I b_Q(u_V(t),u_V(t),\phi)\eta(t)\,dt
   =(u_0,J_Q\phi)_{H_Q}\eta(a)
    +\int_I\mathcal F(t)(\phi)\eta(t)\,dt .
\end{aligned}
\end{equation}
\end{theorem}

\begin{remark}[Identification with classical spaces]
\label{rem:standard-spaces}
Each field in the defining class is globally \(C^1\) and supported in
\(\overline Q\); its second component agrees with the classical derivative in
the interior, and boundary values of the derivative have no effect on its
\(L^2\) class.  An inward dilation of \(Q\), followed by mollification,
approximates each such solenoidal field in
\(H^1\) by elements of \(C^\infty_{c,\sigma}(Q^\circ)\).  Conversely, these
smooth compactly supported fields satisfy the defining conditions after
extension by zero.  The two dense classes therefore have the same graph
closure.  The first-coordinate
map identifies \(V_Q\) isometrically with
\(H^1_{0,\sigma}(Q^\circ)\), equipped with its velocity--gradient graph norm,
and identifies \(H_Q\) with \(L^2_\sigma(Q^\circ)\).  The equivalent
weak-divergence and normal-trace characterizations belong to the classical
density and trace theory; see~\cite[Chapter~I]{Temam1977}.
\end{remark}

\section{Outline of the Proof}\label{sec:proof-outline}

The main analytic ingredients and the logical steps of the proof are
summarized in
\cref{fig:analytic-ingredients,fig:leray-hopf-proof-chain}.

\begin{figure}[htbp]
  \centering
  \begin{tikzpicture}[
    box/.style={
      draw=black!75,
      rounded corners=2pt,
      align=center,
      font=\footnotesize,
      text width=2.35cm,
      minimum height=1.05cm,
      inner sep=3pt
    },
    concrete/.style={box, fill=white},
    abstract/.style={box, fill=black!10},
    flow/.style={-{Stealth[length=2.2mm]}, line width=0.45pt}
  ]
    \node[concrete] (rellich) {Scalar Fourier\\Rellich theorem};
    \node[concrete, right=4mm of rellich] (embedding)
      {Compact embedding\\$J_Q:V_Q\to H_Q$};
    \node[abstract, right=4mm of embedding] (gram)
      {Compact Gram operator\\$J_Q^*J_Q$};
    \node[abstract, right=4mm of gram] (spectral)
      {Spectral theorem\\and Hilbert bases};
    \node[abstract, right=4mm of spectral] (coordinates)
      {Singular-value bases\\and nested finite sets};

    \node[concrete, above=9mm of rellich] (core)
      {Compactly supported $C^1$\\solenoidal fields};
    \node[concrete, above=9mm of gram] (cancellation)
      {Transport\\cancellation};
    \node[concrete, above=9mm of coordinates] (convection)
      {Bounded skew\\convection form};

    \node[concrete, below=9mm of embedding] (lady)
      {Ladyzhenskaya\\estimate on $Q$};
    \node[concrete, below=9mm of spectral] (dual)
      {Uniform $V_Q'$\\derivative bound};

    \draw[flow] (core) -- (cancellation);
    \draw[flow] (cancellation) -- (convection);
    \draw[flow] (rellich) -- (embedding);
    \draw[flow] (embedding) -- (gram);
    \draw[flow] (gram) -- (spectral);
    \draw[flow] (spectral) -- (coordinates);
    \draw[flow] (lady) -- (dual);
  \end{tikzpicture}
  \caption{Analytic ingredients.  White boxes mark constructions and estimates
  first established on rectangles and then used on bounded open domains;
  gray boxes are abstract Hilbert-space arguments.  The rows record transport
  cancellation, the compact-embedding construction, and the time-derivative
  estimate.}
  \label{fig:analytic-ingredients}
\end{figure}
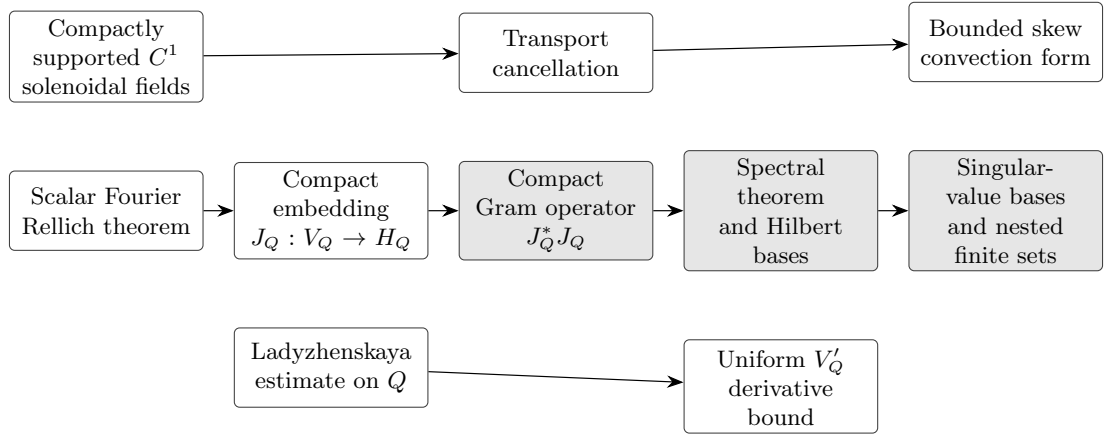

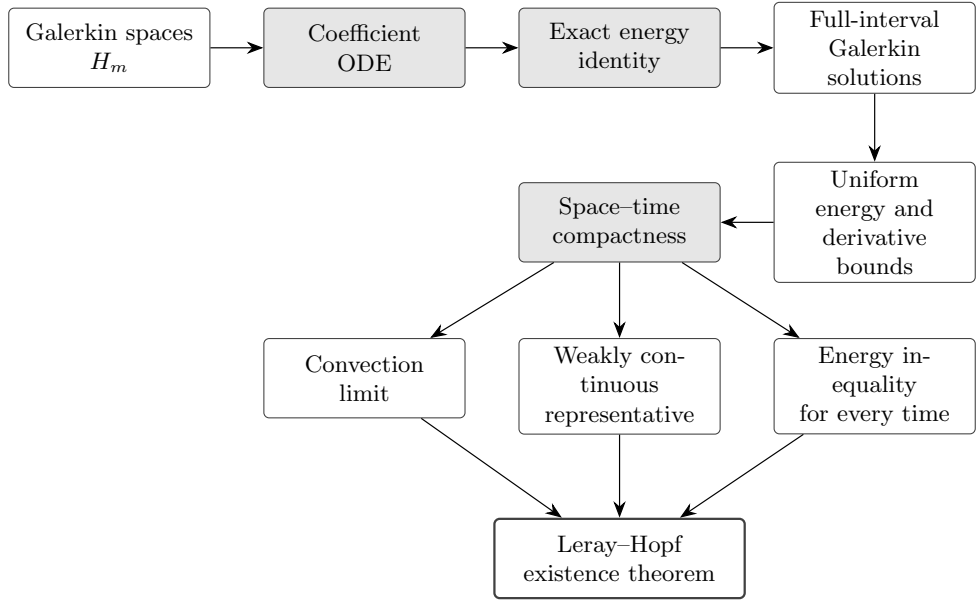
\begin{figure}[htbp]
  \centering
  \begin{tikzpicture}[
    box/.style={
      draw=black!75,
      rounded corners=2pt,
      align=center,
      font=\footnotesize,
      text width=2.45cm,
      minimum height=1.05cm,
      inner sep=3pt
    },
    concrete/.style={box, fill=white},
    abstract/.style={box, fill=black!10},
    conclusion/.style={
      box,
      fill=white,
      line width=0.85pt,
      text width=3.1cm
    },
    flow/.style={-{Stealth[length=2.2mm]}, line width=0.45pt}
  ]
    \node[concrete] (spaces) {Galerkin spaces\\$H_m$};
    \node[abstract, right=7mm of spaces] (ode)
      {Coefficient\\ODE};
    \node[abstract, right=7mm of ode] (identity)
      {Exact energy\\identity};
    \node[concrete, right=7mm of identity] (global)
      {Full-interval\\Galerkin solutions};

    \node[concrete, below=9mm of global] (bounds)
      {Uniform energy and\\derivative bounds};
    \node[abstract, left=7mm of bounds] (compactness)
      {Space--time\\compactness};

    \node[concrete, below left=10mm and 7mm of compactness] (nonlinear)
      {Convection\\limit};
    \node[concrete, below=10mm of compactness] (representative)
      {Weakly continuous\\representative};
    \node[concrete, below right=10mm and 7mm of compactness] (inequality)
      {Energy inequality\\for every time};
    \node[conclusion, below=11mm of representative] (theorem)
      {Leray--Hopf\\existence theorem};

    \draw[flow] (spaces) -- (ode);
    \draw[flow] (ode) -- (identity);
    \draw[flow] (identity) -- (global);
    \draw[flow] (global) -- (bounds);
    \draw[flow] (bounds) -- (compactness);
    \draw[flow] (compactness) -- (nonlinear);
    \draw[flow] (compactness) -- (representative);
    \draw[flow] (compactness) -- (inequality);
    \draw[flow] (nonlinear) -- (theorem);
    \draw[flow] (representative) -- (theorem);
    \draw[flow] (inequality) -- (theorem);
  \end{tikzpicture}
  \caption{Outline of the Leray--Hopf proof.  The exact finite-dimensional
  energy identity extends the local solutions to the full interval and
  yields uniform bounds.  A single subsequence yields the nonlinear limit, the
  weakly continuous representative, and the lower-semicontinuity argument for
  the energy inequality.  These steps prove
  \cref{thm:box-energy-weak-limit}; restriction to bounded open domains and
  consistency across finite horizons give \cref{thm:open-domain-global}.}
  \label{fig:leray-hopf-proof-chain}
\end{figure}

The spaces \(V_Q\) and \(H_Q\) are obtained by completing compactly supported
\(C^1\), divergence-free fields in the velocity--gradient and velocity norms.
Transport integration by parts gives skew-symmetry on this dense subspace,
while continuity extends the convection form to \(V_Q\).  In two
dimensions, the Gagliardo--Nirenberg--Sobolev inequality applied to
\(\lvert u\rvert^2\) gives the Ladyzhenskaya estimate and identifies the
extended form with the classical convection integral.  Compactness of
\(J_Q\) follows by reducing to a scalar zero-trace space and approximating its
embedding in operator norm by low-frequency Fourier projections.

For each \(m\), the paired spectral bases define a lift
\(E_m:H_m\to V_Q\) satisfying
\(J_QE_m=\operatorname{id}_{H_m}\) and a test projection compatible with the
\(H_Q\)-coordinates.  Restricting diffusion and convection to \(H_m\) gives a
finite-dimensional variational problem.  Local Picard solutions continue to
the full interval because skew-symmetry yields an exact energy identity and a
uniform \(H_Q\)-bound.

The energy identity controls the lifted solutions in \(L^2(I;V_Q)\) and their images in
\(L^\infty(I;H_Q)\).  The Ladyzhenskaya estimate bounds the time derivative
in \(L^2(I;V_Q')\), and coordinate compatibility turns this dual estimate
into a uniform \(1/2\)-H\"older modulus for every fixed finite-dimensional
\(H_Q\)-projection.
Finite-dimensional Arzel\`a--Ascoli compactness and the uniform spectral-tail estimate
then give strong \(L^2(I;H_Q)\) convergence.  Hilbert weak compactness gives
the simultaneous weak \(L^2(I;V_Q)\) limit, and continuity of the embedding
identifies its image with the strong \(H_Q\)-limit.

Finally, interpolation yields strong \(L^2(I;L^4(Q))\) convergence, which
permits passage to the limit in the convection term.  Uniform convergence of
all fixed finite projections
determines a weakly continuous representative with the prescribed initial
value.  At each time, the Galerkin value and its time-restricted gradient form
a vector in a Hilbert direct sum.  Applying weak lower semicontinuity to this
vector yields the inequality in \cref{thm:box-energy-weak-limit} from the
exact finite-dimensional energy identity.

For a bounded open set \(\Omega\), zero extension gives the corresponding
compact injection and convection estimates.  Applying
the finite-interval construction on \([0,N]\) for each integer \(N\), then
identifying the weakly continuous representatives on overlaps, yields the
global path in \cref{thm:open-domain-global}.

\section{Functional-Analytic Construction}\label{sec:formalization-design}

Four structures connect the spatial model, finite-dimensional dynamics, and
compactness argument: a compact dense injection, a bounded skew trilinear
form, compatible finite-rank projections, and estimates uniform in the
Galerkin dimension.

\subsection{Graph Closures and the Spatial Model}

The energy space is a closed subspace of the Hilbert product of velocity
and gradient \(L^2\) spaces.  Completeness and the graph norm are inherited
from that product; the two coordinate projections give the velocity embedding
and gradient.  Injectivity of the velocity projection follows from integration
by parts on the defining dense subspace and continuity on the graph closure,
so a zero velocity has zero weak gradient.  The same argument transfers the
divergence constraint.  Consequently, elements of \(V_Q\) have a well-defined
weak gradient and are weakly divergence-free.

Mathlib's closed-subspace and \(L^p\) theory provide the Hilbert structure.
The abstract Galerkin hypotheses are stated in terms of \(J_Q\), \(G_Q\),
their norm identity, compactness, and the convection bounds.  The classical
identification in \cref{rem:standard-spaces} relates the graph model to the
usual Sobolev spaces.

The convection form is first defined on the defining dense subspace and then
extended to the energy completion.  The continuous map
\(R_Q:V_Q\to L^4(Q;\mathbb R^2)\) and the Ladyzhenskaya estimate make this
extension continuous and identify it with the transport integral.
Skew-symmetry extends by density.  Consequently,
the same continuous trilinear form on \(V_Q^3\) appears in every Galerkin
system and in the limiting equation.

\subsection{Spectral Bases and Finite-Dimensional Spaces}

The spectral construction starts from the compact injection \(J:V\to H\).
The positive compact operator \(J^*J\) has an orthonormal basis \(e_i\) of
\(V\); normalizing \(Je_i\) gives an orthonormal basis \(f_i\) of \(H\) and
positive weights
\(\sigma_i\) satisfying \(Je_i=\sigma_i f_i\).  These data construct the
Galerkin spaces without first constructing a differential Stokes operator
or proving elliptic regularity for its eigenfunctions.

The two bases serve different purposes.  The \(H\)-basis gives the kinetic
energy inner product used by the coefficient ODE.  The \(V\)-basis gives
contractive test projections and the gradient bounds.  Their compatibility,
\[
  JP_m^V=P_m^HJ,
\]
shows that the \(V\)- and \(H\)-projections commute with \(J\).  Choosing
nested finite subsets whose union is the full index set gives the same
construction for finite- and infinite-dimensional Hilbert spaces.  The
singular-value construction
depends only on the two Hilbert spaces and the compact dense injection.  For
\(Q\) the injection is \(J_Q\); for \(\Omega\), it is the corresponding
restriction to the closed solenoidal graph space.

\subsection{Time-Dependent Spaces and Compactness}

The interval \(I=[a,b]\) carries restricted Lebesgue measure.  Continuous
Galerkin solutions are bounded and therefore define Bochner \(L^p\) classes.
Continuous linear spatial maps induce continuous maps between these Bochner
spaces.  In particular, the operator induced by \(J_Q\) identifies the
\(L^2(I;H_Q)\)-limit with \(J_Qu_V\), while the weakly continuous
representative agrees with this class almost everywhere.

Every scalar time test in \(C^1(I)\) and its derivative belong to the required
Bochner spaces.  Multiplication by a fixed spatial vector defines the
corresponding test functions in the finite-dimensional identities and in the
limiting weak formulation.

Compactness follows from two estimates.  Each fixed finite projection is
equicontinuous and bounded, so Arzel\`a--Ascoli gives a uniformly convergent
subsequence.
Compactness of \(J\) gives
\(\|J-P_m^HJ\|_{V\to H}\to0\), which controls the high-mode remainder in
\(L^2(I;H)\) by the \(L^2(I;V)\) bound.  A diagonal subsequence therefore
gives strong convergence in \(L^2(I;H)\).  Weak Hilbert-space compactness of
the lifted solutions and a further subsequence give the strong and weak
limits simultaneously.
The limits of the finite-dimensional projections determine the weakly
continuous representative.

The energy inequality uses the same subsequence.  For fixed \(t\),
restriction \(R_t\) to \([a,t]\) is bounded on the time-dependent gradient
space.  The pair
\[
  \bigl(u_n(t),\sqrt{2}\,R_tG_Q u_{V,n}\bigr)
\]
converges weakly in the product of the velocity and gradient Hilbert spaces.
Its squared norm is the kinetic energy plus accumulated dissipation.  Weak lower
semicontinuity of the product norm yields both terms simultaneously and
preserves the initial energy bound at every time.

\subsection{Abstract Hilbert-Space Results and Forcing}

Four arguments are stated independently of the domain: extension of a
trilinear form from a dense subspace, singular-value bases for compact
Hilbert-space injections, continuation of variational Galerkin ODEs from the
energy estimate, and a compactness criterion based on spectral projections.
The scalar
Fourier Rellich argument proves the compact embedding on rectangular
domains.  The two-dimensional specialization enters through the
Ladyzhenskaya estimate and the ensuing uniform \(V_Q'\)-bound on the time
derivatives.

For every \(m\), forcing acts in the fixed dual space \(V_Q'\), so the
finite-dimensional systems, the uniform dual estimate, and the compactness
argument use the same class of forcing terms.  For nonzero forcing,
continuation requires an integrable bound on the forcing work that is
independent of the solution.  The Poincar\'e inequality gives such a bound.
It follows from the Gagliardo--Nirenberg--Sobolev estimate applied to the
field itself rather than to its square.  Since both
coordinate projections are continuous, the inequality defines a closed
subset of the ambient Hilbert product and therefore extends from the defining
smooth class to its closure.
For an admissible forcing, the continuous Riesz representative identifies
the accumulated work with a bounded functional on \(L^2(I;V_Q)\).
Restriction to an initial segment defines another bounded functional on this
space, so weak convergence passes the forcing work to the limit.  Weak lower
semicontinuity of the velocity--gradient norm then yields the forced energy
inequality at every time.

Mathlib supplies the divergence theorem on rectangles, Gagliardo--Nirenberg--Sobolev
inequality, Fourier series and Parseval identity on the multidimensional
torus, compact self-adjoint spectral theorem, Riesz representation,
Picard--Lindel\"of theorem, Arzel\`a--Ascoli theorem, and the underlying
completeness and Bochner-integration theory of \(L^p\) spaces.

\section{Spatial Estimates on the Auxiliary Rectangle}

\subsection{Transport Cancellation}

The energy method rests on the transport integration-by-parts identity.  For
smooth vector fields on \(Q\), with outward unit normal \(n\) on
its faces, the product rule gives the balance
\begin{equation}\label{eq:transport-balance}
  \int_Q (u \cdot \nabla)v \cdot w
  +
  \int_Q v \cdot (u \cdot \nabla)w
  +
  \int_Q (\nabla\cdot u)\, v\cdot w
  =
  \int_{\partial Q} (u\cdot n)\,v\cdot w\,dS .
\end{equation}
When \(u\) is divergence-free and the boundary flux in
\cref{eq:transport-balance} vanishes, this becomes the skew identity
\begin{equation}\label{eq:transport-skew}
  \int_Q (u \cdot \nabla)v \cdot w
  =
  -
  \int_Q v \cdot (u \cdot \nabla)w.
\end{equation}
Taking $w=v$ gives the energy cancellation
\[
  \int_Q (u \cdot \nabla)v \cdot v = 0.
\]
For Navier--Stokes the transported field is the velocity itself, hence the
cubic cancellation.  The abstract Hilbert-space form is
\[
  \ip{\Bil(u,v)}{w} = - \ip{\Bil(u,w)}{v},
\]
and therefore
\[
  \ip{\Bil(u,u)}{u} = 0.
\]

\Cref{prop:box-transport-ibp} gives this identity under the regularity and
face conditions used in the Galerkin construction.

\subsection{Energy Spaces and Compactness}

For incompressible Navier--Stokes, the no-slip model imposes homogeneous
Dirichlet trace on the velocity field.  On \(Q\), zero face values cancel the
boundary flux term in \cref{eq:transport-balance}.

\begin{proposition}[Transport integration by parts on a rectangular domain]
\label{prop:box-transport-ibp}
Let
\[
  Q=\prod_{i=1}^{d}[a_i,b_i]\subset\mathbb R^d,
  \qquad d\geq 1.
\]
Let \(u\colon Q\to\mathbb R^d\) and \(f\colon Q\to\mathbb R\) be continuous
on \(Q\) and Fr\'echet differentiable on \(Q^\circ\), with specified
derivatives \(Du\) and \(Df\).  Suppose \(f\) vanishes on every face of \(Q\)
and the two functions
\[
  x\longmapsto Df(x)[u(x)],
  \qquad
  x\longmapsto f(x)\sum_{i=1}^{d}(Du(x)e_i)_i
\]
are integrable.  Then
\begin{equation}\label{eq:box-scalar-ibp}
  \int_Q Df(x)[u(x)]\,dx
  =-
  \int_Q f(x)\sum_{i=1}^{d}(Du(x)e_i)_i\,dx.
\end{equation}
Consequently, if \(u,v,w\colon Q\to\mathbb R^d\) have the same regularity,
\(v\) and \(w\) vanish on every face, and
\[
  \sum_{i=1}^{d}(Du(x)e_i)_i=0 \qquad (x\in Q),
\]
then, whenever the displayed transport terms are integrable,
\begin{equation}\label{eq:box-vector-ibp}
  \int_Q Dv(x)[u(x)]\cdot w(x)\,dx
  =-
  \int_Q v(x)\cdot Dw(x)[u(x)]\,dx.
\end{equation}
In particular, a field \(u\) vanishing on \(\partial Q\) and satisfying the divergence-free
hypotheses obeys
\begin{equation}\label{eq:box-cubic-cancel}
  \int_Q Du(x)[u(x)]\cdot u(x)\,dx=0.
\end{equation}
\end{proposition}

\begin{proof}
Apply the divergence theorem on \(Q\) to \(fu\).  The boundary flux is zero and
\(\operatorname{div}(fu)=Df[u]+f\operatorname{div}u\), giving
\cref{eq:box-scalar-ibp}.  Taking \(f=v\cdot w\) and using the product rule
and \(\operatorname{div}u=0\) gives \cref{eq:box-vector-ibp}; setting
\(v=w=u\) gives \cref{eq:box-cubic-cancel}.
\end{proof}

\begin{samepage}
\begin{proposition}[Finite-dimensional transport cancellation]
\label{prop:box-variational-cancellation}
Let \(W\) be a finite-dimensional real Hilbert space and let
\[
  \mathcal E\colon W\longrightarrow
    \{Q\to\mathbb R^d\}
\]
be a linear map with an associated linear derivative map \(D\mathcal E\).
Suppose every field in the range of \(\mathcal E\) is continuous on \(Q\), Fr\'echet
differentiable on \(Q^\circ\), divergence-free, and zero on every face.
Let \(b\colon W^3\to\mathbb R\) be a continuous trilinear form satisfying
\[
  b(a,c,e)
  =\int_Q D(\mathcal E c)(x)[\mathcal E a(x)]
      \cdot\mathcal E e(x)\,dx,
\]
with the displayed integrands integrable.  Then
\[
  b(a,c,e)=-b(a,e,c),
  \qquad
  b(a,a,a)=0.
\]
Consequently, every continuous bilinear diffusion form
\(A\colon W^2\to\mathbb R\) satisfying \(A(a,a)\geq0\), together with a
trilinear bound for \(b\), determines a finite-dimensional variational Galerkin
system with transport cancellation.
\end{proposition}

\begin{proof}
Apply \cref{prop:box-transport-ibp} to
\(u=\mathcal E a\), \(v=\mathcal E c\), and \(w=\mathcal E e\).  The resulting
vector identity is the skew relation for \(b\); its diagonal specialization
is the cancellation identity.  The diffusion form, the continuous trilinear
form and its bound, diffusion nonnegativity, and the derived skew relation
verify the hypotheses of the variational Galerkin system.
\end{proof}
\end{samepage}

\begin{proposition}[Energy and velocity spaces on a rectangular domain]
\label{prop:box-energy-closures}
Let \(\mathcal D_\sigma(Q)\) consist of pairs \((u,Du)\) such that
\(u\in C_c^1(\mathbb R^d;\mathbb R^d)\),
\(\operatorname{supp}u\subseteq\overline Q\), \(Du\) agrees
with the Fr\'echet derivative of \(u\) on \(Q^\circ\), \(u\) is zero on every
face, \(\operatorname{tr}Du=0\) on \(Q\), and \(u,Du\in L^2(Q)\).  Define
\begin{align*}
  H_Q
  &:={\overline{\operatorname{span}\{u:u\in\mathcal D_\sigma(Q)\}}}^{L^2(Q;\mathbb R^d)},\\
  V_Q
  &:={\overline{\operatorname{span}\{(u,Du):u\in\mathcal D_\sigma(Q)\}}}
      ^{L^2(Q;\mathbb R^d)\times L^2(Q;\mathbb R^{d\times d})}.
\end{align*}
Then \(H_Q\) and \(V_Q\) are complete real Hilbert spaces.  First-coordinate
projection induces a continuous linear injection
\[
  J_Q\colon V_Q\lhook\joinrel\longrightarrow H_Q,
  \qquad \norm{J_Qv}_{H_Q}\leq\norm{v}_{V_Q},
\]
with dense range, and second-coordinate projection induces
\(G_Q\colon V_Q\to L^2(Q;\mathbb R^{d\times d})\) with the same contraction
bound.  In fact,
\[
  \norm{v}_{V_Q}^2
  =
  \norm{J_Qv}_{H_Q}^2+\norm{G_Qv}_{L^2}^2.
\]
For every \(g\in C_c^\infty(\mathbb R^d)\) and every \(i,j\),
\[
  \ip{(G_Qv)_{ij}}{g}_{L^2(Q)}
  +
  \ip{(J_Qv)_i}{\partial_jg}_{L^2(Q)}
  =0,
\]
and
\[
  \sum_{i=1}^{d}(G_Qv)_{ii}=0
  \qquad\text{in }L^2(Q).
\]
Consequently \(J_Qv\) is weakly divergence-free.  The form
\[
  A_Q(v,w):=\ip{G_Qv}{G_Qw}_{L^2},
  \qquad
  |A_Q(v,w)|\leq\norm{v}_{V_Q}\norm{w}_{V_Q},
\]
and \(A_Q(v,v)=\norm{G_Qv}_{L^2}^2\geq0\).  The pivot pairing defines a
continuous linear map
\[
  \iota_Q\colon H_Q\longrightarrow V_Q',
  \qquad
  (\iota_Qh)(v)=\ip{h}{J_Qv}_{H_Q},
  \qquad
  |(\iota_Qh)(v)|\leq\norm{h}_{H_Q}\norm{v}_{V_Q}.
\]
If \(\mathcal E\) in
\cref{prop:box-variational-cancellation} takes values in
\(\mathcal D_\sigma(Q)\), then
\[
  a\longmapsto (\mathcal Ea,D(\mathcal Ea))
\]
gives a continuous linear lift \(W\to V_Q\).  Restricting \(A_Q\) through this
lift and using the integral convection form produces a finite-dimensional
variational Galerkin system whose diffusion term is the gradient inner
product and whose convection term has the required cancellation.
\end{proposition}

\begin{proof}
Closedness gives the Hilbert-space structures.  Coordinate projection
gives \(J_Q\), \(G_Q\), their contraction bounds, and the graph norm identity;
the image of the defining dense subspace is dense in \(H_Q\).  The diffusion and pivot
bounds follow from Cauchy--Schwarz.

For injectivity, coordinate integration by parts on the defining dense subspace gives
\[
  \int_Q (G_Qv)_{ij}\phi+(J_Qv)_i\,\partial_j\phi\,dx=0
  \qquad(\phi\in C_c^\infty(\mathbb R^d)).
\]
Continuity extends this identity to the graph closure.  If \(J_Qv=0\),
density of scalar smooth tests in \(L^2(Q)\) gives \(G_Qv=0\), hence \(v=0\).
The continuous matrix-trace map vanishes on the defining dense subspace and therefore on its
closure; summing the coordinate identities gives weak incompressibility.
Finally, finite-dimensionality makes the map into the graph space continuous,
and \cref{prop:box-variational-cancellation} gives its transport identity.
\end{proof}

\begin{proposition}[Two-dimensional Ladyzhenskaya estimate and convection form]
\label{prop:box-ladyzhenskaya-convection}
Let \(Q\subset\mathbb R^2\) be a nondegenerate rectangle.  There are a
constant \(C_L\geq0\) and a continuous linear map
\[
  R_Q:V_Q\longrightarrow L^4(Q;\mathbb R^2)
\]
such that \(R_Qu=J_Qu\) almost everywhere and
\begin{equation}\label{eq:box-ladyzhenskaya}
  \norm{R_Qu}_{L^4}^2
  \leq C_L\norm{J_Qu}_{H_Q}\norm{G_Qu}_{L^2}.
\end{equation}
The formula
\begin{equation}\label{eq:box-integral-convection}
  b_Q(u,v,w)
  =
  \int_Q\sum_{i,j=1}^2
    (R_Qu)_j\, (G_Qv)_{ij}\, (R_Qw)_i\,dx
\end{equation}
defines a bounded trilinear form on \(V_Q^3\).  It agrees on the defining dense subspace
with
\[
  \int_Q ((u\cdot\nabla)v)\cdot w\,dx,
\]
and there is a constant \(c_Q>0\) such that
\begin{align}
  |b_Q(u,v,w)|
  &\leq
  c_Q\norm{R_Qu}_{L^4}\norm{G_Qv}_{L^2}
       \norm{R_Qw}_{L^4},\label{eq:box-l4-convection-bound}\\
  b_Q(u,v,w)&=-b_Q(u,w,v),\label{eq:box-convection-skew}\\
  \norm{b_Q(u,u,\cdot)}_{V_Q'}
  &\leq c_QC_L\norm{J_Qu}_{H_Q}\norm{u}_{V_Q}.
  \label{eq:box-diagonal-dual}
\end{align}
\end{proposition}

\begin{proof}
For a compactly supported \(C^1\) vector field \(u\) on \(\mathbb R^2\), apply the
Gagliardo--Nirenberg--Sobolev inequality with exponents \(2\) and \(1\) to
\(q=\lvert u\rvert^2\).  The chain rule and H\"older's inequality give
\[
  \norm{u}_{L^4(\mathbb R^2)}^2
  =\norm{q}_{L^2(\mathbb R^2)}
  \leq C\norm{Dq}_{L^1(\mathbb R^2)}
  \leq 2C\norm{u}_{L^2(\mathbb R^2)}
          \norm{Du}_{L^2(\mathbb R^2)}.
\]
Every finite linear combination of fields in the defining class is \(C^1\), has
support in \(\overline Q\), and has its Fr\'echet derivative represented by
the second component of the graph pair almost everywhere in \(Q\).
Restriction to \(Q\) therefore gives \cref{eq:box-ladyzhenskaya} on the
defining dense subspace, with the fixed
choice of Euclidean matrix norm absorbed into \(C_L\).

The estimate makes the inclusion into \(L^4(Q;\mathbb R^2)\) continuous in
the graph norm, so it extends uniquely by density to \(R_Q\).
Its composition with the continuous inclusion \(L^4(Q)\hookrightarrow
L^2(Q)\) agrees with \(J_Q\) on a dense subspace and hence on all of \(V_Q\).
H\"older's inequality with exponents \(4,2,4\) proves
\cref{eq:box-l4-convection-bound} and the integral representation
\cref{eq:box-integral-convection}.  Transport integration by parts proves
skew-symmetry on the defining dense subspace; continuity and density give
\cref{eq:box-convection-skew} on \(V_Q\).  Finally,
\[
  |b_Q(u,u,\phi)|
  =|b_Q(u,\phi,u)|
  \leq c_Q\norm{R_Qu}_{L^4}^2\norm{G_Q\phi}_{L^2},
\]
and \cref{eq:box-ladyzhenskaya}, together with
\(\norm{G_Qu}_{L^2}\leq\norm{u}_{V_Q}\), proves
\cref{eq:box-diagonal-dual}.
\end{proof}

\begin{proposition}[Poincar\'e inequality and forcing estimate on a rectangle]
\label{prop:box-poincare}
Let \(Q\subset\mathbb R^2\) be a rectangle.  There is a constant
\(C_Q\geq0\) such that
\begin{equation}\label{eq:box-poincare}
  \norm{J_Qv}_{H_Q}\leq C_Q\norm{G_Qv}_{L^2(Q)}
  \qquad (v\in V_Q),
\end{equation}
and consequently
\begin{equation}\label{eq:box-graph-gradient}
  \norm{v}_{V_Q}\leq\sqrt{1+C_Q^2}\,\norm{G_Qv}_{L^2(Q)}
  \qquad (v\in V_Q).
\end{equation}
If \(\mathcal F\in V_Q'\) and \(f\geq0\) satisfies
\(\lvert\mathcal F(w)\rvert\leq f\norm{w}_{V_Q}\) for every \(w\in V_Q\), then
\begin{equation}\label{eq:box-forcing-absorption}
  2\mathcal F(v)
  \leq (1+C_Q^2)f^2+\norm{G_Qv}_{L^2(Q)}^2
  \qquad (v\in V_Q).
\end{equation}
\end{proposition}

\begin{proof}
For a compactly supported \(C^1\) vector field \(u\) on \(\mathbb R^2\), the
Gagliardo--Nirenberg--Sobolev inequality with exponent \(1\) gives
\[
  \norm{u}_{L^2(\mathbb R^2)}\leq C\norm{Du}_{L^1(\mathbb R^2)} .
\]
Every finite linear combination of fields in the defining class is \(C^1\) with support in
\(\overline Q\), so its derivative vanishes off \(\overline Q\) as well, and
H\"older's inequality on \(Q\) gives
\(\norm{Du}_{L^1}\leq\lvert Q\rvert^{1/2}\norm{Du}_{L^2(Q)}\).  The resulting
inequality between the two coordinate projections is a closed condition in
the ambient product \(L^2\) space; it therefore passes from the defining
dense subspace to its closure and proves \cref{eq:box-poincare}.  The graph norm identity of
\cref{prop:box-energy-closures} gives \cref{eq:box-graph-gradient}.  Finally,
\[
  2\mathcal F(v)
  \leq 2f\norm{v}_{V_Q}
  \leq 2\bigl(\sqrt{1+C_Q^2}\,f\bigr)\norm{G_Qv}_{L^2}
  \leq (1+C_Q^2)f^2+\norm{G_Qv}_{L^2}^2 .
\]
\end{proof}

\begin{proposition}[Reduction to a vector-valued Rellich embedding]
\label{prop:box-rellich-reduction}
Let
\[
  \widetilde V_Q
  :=
  \overline{\operatorname{span}
    \{(u,Du):u\in C(Q),\ u\text{ is Fr\'echet differentiable on }Q^\circ,
      \ u|_{\partial Q}=0,\ u,Du\in L^2(Q)\}}
\]
in the same product \(L^2\) space, without imposing incompressibility, and let
\[
  \widetilde J_Q:\widetilde V_Q\longrightarrow L^2(Q;\mathbb R^d)
\]
be first-coordinate projection.  If \(\widetilde J_Q\) is compact, then the
canonical divergence-free map \(J_Q:V_Q\to H_Q\) is compact.
\end{proposition}

\begin{proof}
Forgetting compact support, global \(C^1\) regularity, and incompressibility
maps every field in the defining solenoidal class into the zero-trace space.
Closure therefore gives a continuous inclusion
\(V_Q\hookrightarrow\widetilde V_Q\).  Precomposition of
\(\widetilde J_Q\) with this inclusion is compact.  Its range lies in the
closed subspace \(H_Q\) of the ambient velocity \(L^2\) space.  Restricting
the codomain of a compact map to a closed subspace preserves compactness, and
the resulting map is \(J_Q\).
\end{proof}

\begin{proposition}[Reduction to scalar Rellich compactness]
\label{prop:box-scalar-rellich-reduction}
Let
\[
  W_Q
  :=
  \overline{\operatorname{span}
    \{(f,Df):f\in C(Q),\ f\text{ is Fr\'echet differentiable on }Q^\circ,
      \ f|_{\partial Q}=0,\ f,Df\in L^2(Q)\}}
\]
in \(L^2(Q)\times L^2(Q;\mathbb R^d)\), and let
\[
  S_Q:W_Q\longrightarrow L^2(Q)
\]
be first-coordinate projection.  If \(S_Q\) is compact, then both
\(\widetilde J_Q:\widetilde V_Q\to L^2(Q;\mathbb R^d)\) and
\(J_Q:V_Q\to H_Q\) are compact.
\end{proposition}

\begin{proof}
For each \(i\), extracting the \(i\)-th velocity component and the \(i\)-th
row of the gradient gives a continuous map
\[
  C_i:\widetilde V_Q\longrightarrow W_Q.
\]
The scalar coordinate \(\pi_i\widetilde J_Q=S_QC_i\) is therefore compact.
If \(\iota_i:L^2(Q)\to L^2(Q;\mathbb R^d)\) inserts a scalar field in the
\(i\)-th coordinate, then
\[
  \widetilde J_Q
  =
  \sum_{i=1}^{d}\iota_i\pi_i\widetilde J_Q.
\]
Every summand is compact, hence so is the finite sum.
\Cref{prop:box-rellich-reduction} then gives compactness of \(J_Q\).
\end{proof}

\begin{proposition}[Scalar Rellich compactness on rectangles]
\label{prop:box-scalar-rellich}
Let \(d\geq1\), let
\[
  Q=\prod_{j=1}^{d}[a_j,b_j],
  \qquad a_j<b_j,
\]
and let \(W_Q\) and \(S_Q\) be as in
\cref{prop:box-scalar-rellich-reduction}.  Then
\[
  S_Q:W_Q\longrightarrow L^2(Q)
\]
is compact.  Consequently the full vector graph projection
\(\widetilde J_Q\) and the canonical divergence-free map
\(J_Q:V_Q\to H_Q\) are compact.
\end{proposition}

\begin{proof}
First take \(Q=[0,1]^d\).  Identifying the half-open cube with a fundamental
cell embeds real \(L^2(Q)\) isometrically in complex
\(L^2(\mathbb T^d)\).  For \(k\in\mathbb Z^d\), write \(\widehat f(k)\) for
the corresponding Fourier coefficient.  Coordinate integration by parts for
each field in the defining zero-face class gives
\[
  \widehat{(Df)_j}(k)
  =2\pi i k_j\widehat f(k).
\]
Both sides are continuous linear functionals of the value--derivative pair,
so the identity holds on the closed graph \(W_Q\).

For \(N\geq1\), let
\[
  \Lambda_N=\{k\in\mathbb Z^d:|k_j|\leq N
  \text{ for every }j\}.
\]
If \(k\notin\Lambda_N\), some coordinate satisfies \(|k_j|\geq N\).  The
derivative identity and \(4\pi^2\geq1\) imply
\[
  |\widehat f(k)|^2
  \leq \frac1{N^2}\sum_{j=1}^{d}
    |\widehat{(Df)_j}(k)|^2.
\]
Summing over the complement of \(\Lambda_N\) and applying Parseval gives
\[
  \sum_{k\notin\Lambda_N}|\widehat f(k)|^2
  \leq \frac1{N^2}\sum_{j=1}^{d}
    \norm{(Df)_j}_{L^2(Q)}^2
  \leq \frac1{N^2}\norm{(f,Df)}_{W_Q}^2.
\]
The Fourier projection onto \(\Lambda_N\) has finite-dimensional range, and
the displayed estimate shows that these finite-rank operators converge to
the corresponding torus operator in operator norm, with error at most
\(1/N\).  Hence \(S_Q\) is compact on the unit cube.

For a general rectangular domain, put \(\ell_j=b_j-a_j\),
\(A(y)_j=a_j+\ell_jy_j\), and
\(D=\prod_j\ell_j\).  The normalized pullback
\[
  f\longmapsto \sqrt D\,f\circ A
\]
is an \(L^2\)-isometry.  Its derivative coordinates are
\(\sqrt D\,\ell_j(Df)_j\circ A\), so it defines a continuous map from
the completed graph on \(Q\) to the completed graph on the unit cube.  It
also preserves the vanishing boundary values.  Composition with the compact
unit-cube projection is therefore the normalized pullback of \(S_Q\).
Since the target pullback is an isometry, compactness of the conjugated
operator implies compactness of \(S_Q\) on \(Q\).  The vector and
divergence-free conclusions follow from
\cref{prop:box-scalar-rellich-reduction}.
\end{proof}

\begin{proposition}[Compact Gelfand triple and Galerkin energy bound]
\label{prop:box-gelfand-leray-bound}
Let \(V_Q\), \(H_Q\), and \(J_Q\) be the spaces associated with \(Q\)
from \cref{prop:box-energy-closures}.  Write
\(\mathcal G_Q=L^2(Q;\mathbb R^{d\times d})\), so the closed gradient
coordinate is \(G_Q:V_Q\to\mathcal G_Q\).  Then
\[
  V_Q\xhookrightarrow{\,J_Q\,}H_Q
  \xhookrightarrow{\,\iota_Q\,}V_Q'
\]
is a compact Gelfand triple, where
\[
  (\iota_Qh)(v)=\ip{h}{J_Qv}_{H_Q}.
\]
Both arrows are injective, \(J_Q(V_Q)\) is dense in \(H_Q\), and
\[
  \norm{v}_{V_Q}^2
  =\norm{J_Qv}_{H_Q}^2+\norm{G_Qv}_{\mathcal G_Q}^2.
\]
Consequently, for every finite interval \(I\) and every continuous path
\(v:I\to V_Q\),
\begin{equation}\label{eq:box-time-graph-norm}
  \norm{v}_{L^2(I;V_Q)}^2
  =\norm{J_Qv}_{L^2(I;H_Q)}^2
   +\norm{G_Qv}_{L^2(I;\mathcal G_Q)}^2.
\end{equation}

Let \(W\) be a finite-dimensional real Hilbert space, let
\(E:W\to V_Q\) be a linear lift such that \(J_QE\) is an isometry, and set
\[
  a_W(x,y)=\ip{G_QEx}{G_QEy}_{\mathcal G_Q}.
\]
If a continuous function \(u:I\to W\) satisfies
\[
  \sup_{t\in I}\norm{u(t)}_W\leq R,
  \qquad
  \int_I a_W(u(t),u(t))\,dt\leq \Delta^2,
\]
then
\begin{equation}\label{eq:box-galerkin-lift-bound}
  \norm{Eu}_{L^2(I;V_Q)}
  \leq \sqrt{|I|R^2+\Delta^2}.
\end{equation}
\end{proposition}

\begin{proof}
Compactness of \(J_Q\) is \cref{prop:box-scalar-rellich}; injectivity, dense
range, and the graph identity follow from the completed graph construction.
If \(\iota_Qh=0\), then \(h\) is orthogonal to the dense range of \(J_Q\), hence
\(h=0\).  Integrating the pointwise graph identity gives
\cref{eq:box-time-graph-norm}.  Since \(J_QE\) is an isometry,
\(
  \norm{J_QEu}_{L^2(I;H_Q)}^2\leq |I|R^2
\), while the definition of \(a_W\) gives
\(
  \norm{G_QEu}_{L^2(I;\mathcal G_Q)}^2
  =\int_I a_W(u,u)\,dt\leq\Delta^2
\).
Equation \cref{eq:box-galerkin-lift-bound} follows from
\cref{eq:box-time-graph-norm}.
\end{proof}

\begin{proposition}[Singular-value bases for a compact embedding]
\label{prop:compact-embedding-singular-bases}
Let \(V\) and \(H\) be separable complete real Hilbert spaces, and let
\(J:V\to H\) be compact, injective, contractive, and have dense range.  There
exist an at most countable index set \(K\), Hilbert bases
\((e_i)_{i\in K}\) of \(V\) and \((f_i)_{i\in K}\) of \(H\), numbers
\(\sigma_i\in(0,1]\), and finite sets
\[
  K_0\subseteq K_1\subseteq\cdots,
  \qquad
  \bigcup_{m\geq0}K_m=K,
\]
such that
\begin{equation}\label{eq:compact-embedding-singular-bases}
  Je_i=\sigma_i f_i,
  \qquad
  \ip{Jv}{f_i}_H=\sigma_i\ip{v}{e_i}_V
  \quad(v\in V,\ i\in K).
\end{equation}
If \(P_m^V\) and \(P_m^H\) are the orthogonal projections onto the spans of
\(\{e_i:i\in K_m\}\) and \(\{f_i:i\in K_m\}\), then
\[
  P_m^Vv\longrightarrow v,\qquad P_m^Hh\longrightarrow h,
\]
and
\begin{equation}\label{eq:compact-embedding-spectral-tail}
  \norm{J-P_m^HJ}_{V\to H}\longrightarrow0.
\end{equation}
\end{proposition}

\begin{proof}
The compact positive operator \(T=J^*J\) has trivial kernel.  Its spectral
theorem gives an orthonormal eigenbasis \(e_i\) with positive eigenvalues
\(\lambda_i\).  Set \(\sigma_i=\sqrt{\lambda_i}\) and
\(f_i=\sigma_i^{-1}Je_i\).  The adjoint identity proves orthonormality;
density of the range of \(J\) proves completeness in \(H\).
Contractivity gives \(\sigma_i\leq1\).
By separability, there are nested finite subsets whose union is the full index
set.  The corresponding
orthogonal projectors converge strongly to the identity and are
contractive; convergence is uniform on the compact closure of the image of
the unit ball under \(J\).  This gives
\cref{eq:compact-embedding-spectral-tail}.
\end{proof}

\begin{proposition}[Spectral Galerkin spaces on a rectangular domain]
\label{prop:box-spectral-galerkin-spaces}
The spaces \(V_Q\) and \(H_Q\) are separable as closed subspaces of separable
finite products of \(L^2\) spaces.  Therefore the preceding proposition
applies to \(J_Q\).
Choose the spectral data obtained by applying
\cref{prop:compact-embedding-singular-bases} to
\(J_Q:V_Q\to H_Q\), and denote them by
\((e_i)_{i\in K}\), \((f_i)_{i\in K}\), \((\sigma_i)_{i\in K}\), and
\((K_m)_{m\geq0}\).  Set
\[
  g_i=\sigma_i^{-1}e_i,
  \qquad
  H_m=\operatorname{span}\{f_i:i\in K_m\}.
\]
There is a continuous lift \(E_m:H_m\to V_Q\) such that
\[
  J_QE_mx=x,
  \qquad
  \norm{J_QE_mx}_{H_Q}=\norm{x}_{H_m}.
\]
The spaces \(H_m\) are nested and have dense union in \(H_Q\).  Let
\(Q_m:V_Q\to H_m\) be the orthogonal projection of \(J_Qv\) onto \(H_m\).
Then
\begin{equation}\label{eq:energy-test-projection}
  E_mQ_mv
  =\sum_{i\in K_m}\ip{v}{e_i}_{V_Q}e_i
  =:P_m^Vv,
  \qquad
  P_m^Vv\longrightarrow v\quad\text{in }V_Q.
\end{equation}
Then, for every \(x\in H_m\) and every \(i\in K\),
\begin{equation}\label{eq:box-spectral-coordinate}
  \ip{x}{Q_m g_i}_{H_m}=\ip{J_QE_mx}{f_i}_{H_Q}.
\end{equation}
If \(c:V_Q^3\to\mathbb R\) is a bounded trilinear form satisfying
\(c(u,v,w)=-c(u,w,v)\), then the pullbacks of \(c\), the gradient diffusion
form, and a fixed \(V_Q'\)-valued forcing define a finite-dimensional
variational Galerkin problem on every \(H_m\).  Its convection form is skew
and its diffusion form is nonnegative.
\end{proposition}

\begin{proof}
For \(x\in H_m\), define
\[
  E_mx=\sum_{i\in K_m}\ip{x}{f_i}_{H_Q}g_i.
\]
Since \(J_Qg_i=f_i\), orthonormality gives \(J_QE_mx=x\) and the norm
identity.  The finite-dimensional spans are nested and have dense union.
Equation \cref{eq:compact-embedding-singular-bases} gives
\[
  E_mQ_mv
  =\sum_{i\in K_m}\ip{v}{e_i}_{V_Q}e_i.
\]
Completeness of the energy basis proves
\cref{eq:energy-test-projection}; in particular, \(E_mQ_m\) is a contraction.
Orthogonality of the \(H_Q\)-projector gives
\cref{eq:box-spectral-coordinate}.  Composing each argument of the diffusion
and convection forms with \(E_m\) preserves diffusion nonnegativity and
convection skew-symmetry.  Riesz representation on \(H_m\) produces the
right-hand side of the finite-dimensional ODE.
\end{proof}

\section{Variational Galerkin Approximations}

For each \(m\), Riesz representation turns the diffusion, convection, and
forcing forms on \(H_m\) into a coefficient ODE.
\Cref{thm:variational-finite-dimensional} gives local existence, the exact energy
identity, continuation from the energy estimate, and a uniform
\(V_Q'\)-derivative estimate.

\subsection{Finite-Dimensional Variational Solutions}

Let \(W\) be a finite-dimensional Galerkin space equipped with the
\(H\)-inner product.  Let
\[
  a:W\times W\to\R,\qquad
  b:W\times W\times W\to\R,\qquad
  \mathcal F:[t_{\min},t_{\max}]\to W'
\]
where \(a\) and \(b\) are bounded multilinear forms and \(\mathcal F\) is
a forcing path.  Assume
\[
  a(w,w)\geq 0,\qquad
  b(u,v,w)=-b(u,w,v),\qquad
  \|b(u,v,\cdot)\|_{W'}\leq C_b\|u\|_H\|v\|_H .
\]
The constant \(C_b\) may depend on \(m\); it is used only for
finite-dimensional local existence.  Uniform time control is obtained later
from the \(V_Q'\)-estimate on the lifted solutions.
Riesz representation defines \(R(t,u)\in W\) by
\begin{equation}\label{eq:variational-rhs}
  \ip{R(t,u)}{v}_H
  =
  -a(u,v)-b(u,u,v)+\mathcal F(t)(v)
  \qquad(v\in W).
\end{equation}

\begin{theorem}[Finite-dimensional variational Galerkin theorem]
\label{thm:variational-finite-dimensional}
Let \(u_0\in W\), let \(\mathcal F\) be continuous in the dual norm, and assume
\(\|\mathcal F(t)\|_{W'}\leq M\) on
\([t_{\min},t_{\max}]\).  Fix
\(t_0\in[t_{\min},t_{\max}]\) and \(\rho\geq0\), and set
\[
  R_\rho=\|u_0\|_H+\rho,\qquad
  L_\rho=\|a\|R_\rho+C_bR_\rho^2+M.
\]
If
\[
  L_\rho
  \max\{t_{\max}-t_0,t_0-t_{\min}\}\leq\rho,
\]
then \(\dot u=R(t,u)\), \(u(t_0)=u_0\), has a solution on
\([t_{\min},t_{\max}]\).  For every \(T\in[t_0,t_{\max}]\), such a solution
satisfies
\begin{equation}\label{eq:variational-energy}
  \|u(T)\|_H^2+2\int_{t_0}^T a(u,u)\,dt
  =
  \|u_0\|_H^2+2\int_{t_0}^T\mathcal F(t)(u(t))\,dt .
\end{equation}
More generally, if \(\|x\|_H\leq R_0\), the same conclusion with initial
datum \(x\) holds whenever
\[
  \bigl(\|a\|(R_0+\rho)+C_b(R_0+\rho)^2+M\bigr)
  \max\{t_{\max}-t_0,t_0-t_{\min}\}\leq\rho .
\]
Thus the initial-value problem has a uniform interval of local existence for
all initial data in the \(H\)-ball of radius \(R_0\).

If \(g\) is integrable on \([t_0,T]\) and, for \(t\in[t_0,T]\),
\[
  2\mathcal F(t)(u(t))\leq g(t)+a(u(t),u(t)),
\]
then
\begin{equation}\label{eq:variational-apriori}
  \|u(T)\|_H^2+\int_{t_0}^T a(u,u)\,dt
  \leq
  \|u_0\|_H^2+\int_{t_0}^T g(t)\,dt .
\end{equation}

Suppose now that \(g\geq0\) is integrable on
\([t_{\min},t_{\max}]\), that
\[
  2\mathcal F(t)(x)\leq g(t)+a(x,x)
  \qquad
  (t\in[t_{\min},t_{\max}],\ x\in W),
\]
and that \(R_0\geq0\) satisfies
\[
  \|u_0\|_H^2+\int_{t_{\min}}^{t_{\max}}g(t)\,dt\leq R_0^2.
\]
Then the initial-value problem with \(u(t_{\min})=u_0\) has a solution on the
entire interval \([t_{\min},t_{\max}]\), and
\(\|u(t)\|_H\leq R_0\) throughout that interval.

Let \(Q:V\to W\) be a continuous Galerkin test map, let
\(c_A,c_B,h(t),d(t),f(t)\) be nonnegative, and define
\(\mathcal D(t)\in V'\) by
\[
  \mathcal D(t)(\phi)
  =
  -a(u,Q\phi)-b(u,u,Q\phi)+\mathcal F(t)(Q\phi).
\]
If the three terms satisfy
\begin{align*}
  |a(u,Q\phi)|&\leq c_A d(t)\|\phi\|_V,\\
  |b(u,u,Q\phi)|&\leq c_B h(t)d(t)\|\phi\|_V,\\
  |\mathcal F(t)(Q\phi)|&\leq f(t)\|\phi\|_V,
\end{align*}
then
\begin{equation}\label{eq:variational-dual-derivative}
  \|\mathcal D(t)\|_{V'}
  \leq c_A d(t)+c_Bh(t)d(t)+f(t).
\end{equation}
If \(h(t)\leq H_*\), then
\begin{equation}\label{eq:variational-dual-square}
  \|\mathcal D(t)\|_{V'}^2
  \leq 2(c_A+c_BH_*)^2d(t)^2+2f(t)^2.
\end{equation}
Moreover, the finite-dimensional solution satisfies
\[
  \frac{d}{dt}\ip{u(t)}{Q\phi}_H=\mathcal D(t)(\phi)
\]
on its interval of existence.
\end{theorem}

\begin{proof}
Riesz representation applied to the right side of
\cref{eq:variational-rhs} defines \(R\).  Bilinearity gives
\[
  b(u,u,\cdot)-b(v,v,\cdot)
  =
  b(u-v,u,\cdot)+b(v,u-v,\cdot),
\]
so \(R(t,\cdot)\) is Lipschitz on the \(H\)-ball of radius \(R_\rho\) with
constant \(\|a\|+2C_bR_\rho\), and its norm there is bounded by \(L_\rho\).
Picard--Lindel\"of gives the solution under the displayed interval
condition.  Replacing \(\|u_0\|_H\) by the common bound \(R_0\) gives a
uniform local-existence time.  Differentiating \(\|u\|_H^2\), using
\(b(u,u,u)=0\), and integrating gives
\cref{eq:variational-energy}.  The pointwise work estimate absorbs one copy of
the diffusion integral and gives \cref{eq:variational-apriori}.  Under the
hypotheses of the global continuation statement, the estimate on every
partial interval keeps the solution in the ball of radius \(R_0\).  A
uniform local-existence time is
chosen as
\[
  \delta=
  \frac{1}{\|a\|(R_0+1)+C_b(R_0+1)^2+M+1}>0.
\]
Uniqueness permits local solutions with matching endpoint data to be
concatenated into a maximal solution.  If its right endpoint were below
\(t_{\max}\), the uniform local-existence time would extend it, a
contradiction.
Composition of the variational right-hand side with \(Q\), followed by the triangle inequality,
gives \cref{eq:variational-dual-derivative}.  The uniform bound on \(h\) and
the elementary inequality \((x+y)^2\leq2x^2+2y^2\) give
\cref{eq:variational-dual-square}; testing the coefficient ODE by \(Q\phi\)
gives the derivative identity.
\end{proof}

\begin{proposition}[Two-dimensional spectral dual estimate]
\label{prop:box-spectral-dual-l2}
Let \(Q\subset\mathbb R^2\) be a rectangle, and let
\((H_m,E_m,Q_m)\) be the data associated with one of the spectral Galerkin
spaces in \cref{prop:box-spectral-galerkin-spaces}.  Equip \(H_m\) with the
diffusion form and the integral convection form
\(b_Q\) of \cref{prop:box-ladyzhenskaya-convection}.  Let
\(\mathcal F:[a,b]\to V_Q'\) be strongly measurable, let
\(U_m:[a,b]\to H_m\), and define
\[
  \mathcal D_m(t)(\phi)
  =
  -A_Q(E_mU_m(t),E_mQ_m\phi)
  -b_Q(E_mU_m(t),E_mU_m(t),E_mQ_m\phi)
  +\mathcal F(t)(E_mQ_m\phi).
\]
Suppose \(a\leq b\), \(f\in L^2(a,b)\), \(f\geq0\), and
\begin{align*}
  |\mathcal F(t)(\phi)|&\leq f(t)\norm{\phi}_{V_Q},\\
  \norm{J_QE_mU_m(t)}_{H_Q}&\leq H_*,\\
  \int_a^b\norm{E_mU_m(t)}_{V_Q}^2\,dt&\leq R_*^2,\\
  \int_a^b|f(t)|^2\,dt&\leq F_*^2,
\end{align*}
where \(H_*\geq0\) and \(E_mU_m\in L^2(a,b;V_Q)\).  Then
\(\mathcal D_m\in L^2(a,b;V_Q')\), and
\begin{equation}\label{eq:box-spectral-dual-square}
  \int_a^b\norm{\mathcal D_m(t)}_{V_Q'}^2\,dt
  \leq
  2(1+c_QC_LH_*)^2R_*^2+2F_*^2.
\end{equation}
\end{proposition}

\begin{proof}
Since \(E_mQ_m\) is a contraction,
\(\norm{E_mQ_m\phi}_{V_Q}\leq\norm{\phi}_{V_Q}\).  The gradient form and
\cref{eq:box-diagonal-dual} therefore give
\[
  \norm{\mathcal D_m(t)}_{V_Q'}
  \leq
  (1+c_QC_LH_*)\norm{E_mU_m(t)}_{V_Q}+f(t).
\]
The diffusion and convection terms are strongly measurable because
\(E_mU_m\) is strongly measurable and the corresponding maps are continuous.
Strong measurability of \(\mathcal F\) is preserved by composition with
\(E_mQ_m\).
Thus \(\mathcal D_m\in L^2(a,b;V_Q')\).
Squaring with \((x+y)^2\leq2x^2+2y^2\) and integrating gives
\cref{eq:box-spectral-dual-square}.
\end{proof}

\begin{proposition}[Unforced spectral Galerkin solutions]
\label{prop:box-zero-forcing-spectral-family}
Let the paired spectral bases and finite-dimensional spaces be those constructed in
\cref{prop:compact-embedding-singular-bases,prop:box-spectral-galerkin-spaces}.
For \(a\leq b\) and \(u_0\in H_Q\), define
\[
  u_{0,m}=P_mu_0\in H_m,
  \qquad
  \rho_{a,b}(u_0)
  =\sqrt{1+\lvert[a,b]\rvert}\,\norm{u_0}_{H_Q}.
\]
The unforced two-dimensional variational problem on every \(H_m\) has a
solution \(U_m:[a,b]\to H_m\) with \(U_m(a)=u_{0,m}\), and
\begin{align}
  \sup_{t\in[a,b]}\norm{U_m(t)}_{H_m}
    &\leq\norm{u_0}_{H_Q},\label{eq:box-zero-h-bound}\\
  \int_a^b\norm{E_mU_m(t)}_{V_Q}^2\,dt
    &\leq\rho_{a,b}(u_0)^2.\label{eq:box-zero-energy-bound}
\end{align}
Moreover,
\[
  J_QE_mu_{0,m}=P_mu_0\longrightarrow u_0
  \qquad\text{strongly in }H_Q.
\]
The family \((J_QE_mU_m)_m\) has a subsequence converging strongly in
\(L^2(a,b;H_Q)\).
\end{proposition}

\begin{proof}
Since \(u_{0,m}=P_mu_0\) and \(P_m\) is an orthogonal projection,
\[
  \norm{u_{0,m}}_{H_Q}\leq\norm{u_0}_{H_Q}.
\]
Completeness of the basis gives strong convergence of \(P_mu_0\) to \(u_0\).

For each \(m\) the forcing vanishes, so
\cref{thm:variational-finite-dimensional} applies with \(g=0\).  The energy estimate
extends \(U_m\) to the full interval and gives
\cref{eq:box-zero-h-bound} and bounds the dissipation integral by
\(\norm{u_0}_{H_Q}^2\).  The exact graph
identity
\[
  \norm{E_mU_m}_{V_Q}^2
  =
  \norm{J_QE_mU_m}_{H_Q}^2
  +\norm{G_QE_mU_m}_{L^2}^2
\]
then gives \cref{eq:box-zero-energy-bound}.

Apply \cref{prop:box-spectral-dual-l2} with \(f=0\),
\(H_*=\norm{u_0}_{H_Q}\), and \(R_*=\rho_{a,b}(u_0)\).  The coordinate
identity \cref{eq:box-spectral-coordinate} converts the resulting uniform dual
bound into a uniform \(1/2\)-H\"older estimate for each fixed projection.
Compactness of
\(J_Q\), the uniform \(V_Q\)-bound, and the uniform \(H_Q\)-bound give a strongly
convergent subsequence in \(L^2(a,b;H_Q)\).
\end{proof}

\begin{proposition}[Finite-mode regularity from a dual \(L^2\) derivative]
\label{prop:dual-l2-holder}
Let \(I=[a,b]\), let \(H\) be a real Hilbert space, let \(K\) be an at most
countable index set, and let \((f_i)_{i\in K}\) be a Hilbert basis of \(H\).
Choose nested finite sets
\[
  K_0\subseteq K_1\subseteq\cdots,
  \qquad
  \bigcup_{m\geq0}K_m=K,
\]
and let \(P_m\) be orthogonal projection onto
\(\operatorname{span}\{f_i:i\in K_m\}\).  Let
\((W_n)_{n\in\mathbb N}\) be finite-dimensional real Hilbert spaces, let
\(J:V\to H\) be continuous and linear, and let
\((\psi_i)_{i\in K}\) lie in a normed test space \(X\).  For each \(n\),
suppose that \(u_n:I\to W_n\) is a finite-dimensional variational solution,
\(v_n:I\to V\) is a lift, \(Q_n:X\to W_n\) is continuous and linear, and
\(\mathcal D_n\in L^2(I;X')\) satisfies
\begin{align}
  \ip{u_n(t)}{Q_n\psi_i}_{W_n}
    &=\ip{Jv_n(t)}{f_i}_H,\label{eq:coordinate-compatibility}\\
  \frac{d}{dt}\ip{u_n(t)}{Q_n\psi_i}_{W_n}
    &=\mathcal D_n(t)(\psi_i).\label{eq:coordinate-derivative}
\end{align}
If
\[
  \int_I\norm{\mathcal D_n(t)}_{X'}^2\,dt\leq M^2
  \qquad(n\in\mathbb N),
\]
then, with
\[
  \Gamma_m=\sum_{i\in K_m}\norm{\psi_i}_X,
\]
one has
\begin{align}
  \int_I\norm{\partial_t(P_mJv_n)(t)}_H^2\,dt
    &\leq (\Gamma_mM)^2,\label{eq:finite-mode-derivative-l2}\\
  \norm{P_mJv_n(t)-P_mJv_n(s)}_H
    &\leq \Gamma_mM\lvert t-s\rvert^{1/2}
      \qquad(s,t\in I).\label{eq:finite-mode-holder}
\end{align}
In particular, for each \(m\), the family \(\{P_mJv_n\}_n\) is
equicontinuous.
\end{proposition}

\begin{proof}
The finite-rank projector has the coordinate formula
\[
  P_mx=\sum_{i\in K_m}\ip{x}{f_i}_Hf_i.
\]
Equations \cref{eq:coordinate-compatibility,eq:coordinate-derivative} and
differentiation of the finite sum therefore give
\[
  \partial_t(P_mJv_n)(t)
    =\sum_{i\in K_m}\mathcal D_n(t)(\psi_i)f_i.
\]
The triangle inequality, the dual norm estimate, and
\(\norm{f_i}_H=1\) imply
\[
  \norm{\partial_t(P_mJv_n)(t)}_H
  \leq
  \sum_{i\in K_m}\lvert\mathcal D_n(t)(\psi_i)\rvert
  \leq \Gamma_m\norm{\mathcal D_n(t)}_{X'}.
\]
Squaring and integrating proves
\cref{eq:finite-mode-derivative-l2}.  The interval fundamental theorem of
calculus and Cauchy--Schwarz yield, by symmetry for \(s\leq t\),
\[
  \norm{P_mJv_n(t)-P_mJv_n(s)}_H
  \leq
  \lvert t-s\rvert^{1/2}
  \left(
    \int_s^t\norm{\partial_r(P_mJv_n)(r)}_H^2\,dr
  \right)^{1/2},
\]
which gives \cref{eq:finite-mode-holder}.
\end{proof}
\section{Compactness and the Leray--Hopf Limit}

The compactness argument uses the orthogonal projections associated with
nested finite subsets of a Hilbert basis.  The same construction applies in
both finite- and infinite-dimensional Hilbert spaces.

\begin{proposition}[Galerkin projectors from a Hilbert basis]
\label{prop:hilbert-projectors}
For a Hilbert basis \((b_i)_{i\in K}\) and nested finite sets
\(K_n\uparrow K\), the spaces \(H_n=\operatorname{span}\{b_i:i\in K_n\}\)
are nested and finite-dimensional.  Their orthogonal projectors satisfy
\(P_n^2=P_n\) and \(P_nx\to x\) for every \(x\in H\).
\end{proposition}
\begin{proof}
The union of the partial spans is dense, and the projectors are contractive.
\end{proof}

\subsection{Spectral Compactness in Space--Time}

Let \(I=[0,T]\), let \(V\) be a Banach space, let \(H\) be a Hilbert space,
and let \(J:V\to H\) be compact and linear.  Write \(P_m\) for the
Hilbert-basis projectors of \cref{prop:hilbert-projectors}.  For continuous
lifts \(v_n:I\to V\), set
\[
  u_n(t)=Jv_n(t).
\]
The derivative hypotheses of \cref{prop:dual-l2-holder} imply the
finite-mode hypothesis of \cref{thm:spectral-l2-compactness} with exponent
\(\alpha=1/2\).

\begin{theorem}[Spectral compactness in \(L^2(I;H)\)]
\label{thm:spectral-l2-compactness}
Suppose
\[
  \sup_n\norm{v_n}_{L^2(I;V)}\leq R_V,
  \qquad
  \sup_n\norm{u_n}_{L^\infty(I;H)}\leq R_H.
\]
Assume that there is \(\alpha>0\) such that for every \(m\)
there is \(C_m<\infty\) such that
\[
  \norm{P_m u_n(t)-P_m u_n(s)}_H
  \leq C_m |t-s|^\alpha
  \qquad(n\in\mathbb N,\ s,t\in I).
\]
Then there exist indices \(n_j\to\infty\) and \(u\in L^2(I;H)\) such that
\[
  u_{n_j}\longrightarrow u
  \qquad\text{strongly in }L^2(I;H).
\]
\end{theorem}

\begin{proof}
The orthogonal projectors satisfy \(\norm{P_m}\leq1\) and converge strongly
to the identity.  Compactness of \(J\) implies
\[
  \norm{J-P_mJ}_{V\to H}\longrightarrow0.
\]
Indeed, the image under \(J\) of the closed unit ball is totally bounded.
A finite \(\varepsilon\)-net reduces uniform convergence on this image to
pointwise convergence at finitely many centers; contractivity controls the
errors between a point and its center.

Consequently,
\[
  \norm{u_n-P_m u_n}_{L^2(I;H)}
  \leq \norm{J-P_mJ}_{V\to H}\norm{v_n}_{L^2(I;V)}
  \leq R_V\norm{J-P_mJ}_{V\to H},
\]
uniformly in \(n\).  For fixed \(m\), the paths \(P_m u_n\) take values in a
bounded subset of the finite-dimensional space \(H_m\) and share the stated
H\"older modulus.  Arzel\`a--Ascoli makes their closure compact in
\(C(I;H_m)\), hence also in \(L^2(I;H)\).  Together with the uniform tail
estimate, this proves that \(\{u_n\}\) is relatively compact in
\(L^2(I;H)\).
\end{proof}
\subsection{The Finite-Horizon Weak-Solution Limit}\label{sec:proof-main}

\begin{proof}[Proof of \cref{thm:box-energy-weak-limit}]
\Cref{prop:compact-embedding-singular-bases} constructs paired
singular-value bases for \(J_Q\), and
\cref{prop:box-spectral-galerkin-spaces} gives the associated
finite-dimensional subspaces of \(H_Q\), lifts to \(V_Q\), and test
projections.  Choose the subsequence \((m_k)\) supplied by
\cref{thm:spectral-l2-compactness}, and write
\[
  u_{V,k}=E_{m_k}U_{m_k},
  \qquad
  u_{H,k}=J_Qu_{V,k},
  \qquad
  \phi_k=E_{m_k}Q_{m_k}\phi=P_{m_k}^V\phi.
\]
By \cref{eq:energy-test-projection}, \(\phi_k\to\phi\) strongly in \(V_Q\).
For every \(k\), the Galerkin solution satisfies the time-tested identity
obtained by integration by parts.  The projection identity expresses the
velocity term using the original test \(\phi\), while diffusion and convection
use \(\phi_k\); the initial term contains the orthogonal projection of \(u_0\).

After passing to a further subsequence, the uniform \(L^2(I;V_Q)\) bound and
Hilbert weak compactness give
\[
  u_{V,k}\rightharpoonup u_V\quad\text{in }L^2(I;V_Q),
  \qquad
  u_{H,k}\to u_H\quad\text{in }L^2(I;H_Q).
\]
Continuity of the operator induced by \(J_Q\) on the Bochner spaces identifies
\(J_Qu_V=u_H\).
Weak lower semicontinuity gives the asserted \(L^2(I;V_Q)\) bound.  Since
the time-dependent gradient map is continuous and linear,
\[
  G_Qu_{V,k}\rightharpoonup G_Qu_V
  \quad\text{in }L^2(I;L^2(Q;\mathbb R^{2\times2})).
\]
Strong convergence in \(L^2(I;H_Q)\) implies convergence in
measure; the uniform pointwise \(H_Q\)-estimate then passes to the essential
supremum and yields
\[
  \norm{u_H}_{L^\infty(I;H_Q)}
  \leq\norm{u_0}_{H_Q}.
\]

A diagonal subsequence gives uniform convergence on \(I\) of every fixed
finite-dimensional projection of the \(H_Q\)-valued solutions.  The uniform \(H_Q\)-bound
and convergence on the dense union of finite spectral spaces imply, for every
\(t\in I\),
\[
  u_{H,k}(t)\rightharpoonup u_c(t)\quad\text{in }H_Q.
\]
For each \(y\in H_Q\), the functions
\[
  t\longmapsto (u_c(t),P_my)_{H_Q}
\]
are continuous and converge uniformly as \(m\to\infty\) to
\(t\mapsto(u_c(t),y)_{H_Q}\).  Hence \(u_c\) is weakly continuous.  An
almost-everywhere extraction from the strong \(L^2(I;H_Q)\) convergence
identifies \(u_c\) with \(u_H\) almost everywhere.  At the left endpoint the
Galerkin values are the orthogonal projections of \(u_0\), so
\(u_c(a)=u_0\).  Pointwise weak lower semicontinuity gives
\(\norm{u_c(t)}_{H_Q}\leq\norm{u_0}_{H_Q}\).

Fix \(t\in I\), and let \(R_t\) restrict a time-dependent gradient to
\([a,t]\).  The direct-sum vectors
\[
  X_k(t)=
  \bigl(u_{H,k}(t),\sqrt2\,R_tG_Qu_{V,k}\bigr)
\]
converge weakly in
\[
  H_Q\oplus_2 L^2([a,t];L^2(Q;\mathbb R^{2\times2}))
\]
to
\[
  X(t)=\bigl(u_c(t),\sqrt2\,R_tG_Qu_V\bigr).
\]
The exact finite energy identity and contractivity of the initial projection
give \(\norm{X_k(t)}\leq\norm{u_0}_{H_Q}\).  Weak lower semicontinuity
therefore yields
\[
  \norm{u_c(t)}_{H_Q}^2
  +2\int_a^t\norm{G_Qu_V(s)}_{L^2(Q)}^2\,ds
  \leq\norm{u_0}_{H_Q}^2.
\]
Since \(t\) was arbitrary, the energy inequality holds at every time.

Let \(R_Q:V_Q\to L^4(Q;\mathbb R^2)\) be the map from the Ladyzhenskaya
estimate.
Applied to \(u_{V,k}-u_V\), the space--time interpolation estimate gives
\[
  \norm{R_Q(u_{V,k}-u_V)}_{L^2(I;L^4)}^2
  \leq C_L
  \norm{J_Q(u_{V,k}-u_V)}_{L^2(I;H_Q)}
  \norm{u_{V,k}-u_V}_{L^2(I;V_Q)}.
\]
The first factor tends to zero and the second is bounded, so the velocities
converge strongly in \(L^2(I;L^4)\).  Skew-symmetry moves the spatial
derivative in
\(b_Q(u_{V,k},u_{V,k},\phi_k)\eta\) onto the test field.  The two velocity
factors converge strongly in \(L^2(I;L^4)\), while
\(\eta\,G_Q\phi_k\to\eta\,G_Q\phi\) in
\(L^\infty(I;L^2)\).  H\"older's inequality therefore passes the convection
integral.  Weak \(L^2(I;V_Q)\) convergence passes diffusion against the same
strongly converging tests, and strong convergence of the initial
projections passes the initial pairing.  Passing to the limit in the
finite-dimensional identities gives the weak equation.
\end{proof}

\section{Bounded Open Domains and Global Time}
\label{sec:open-domain-global}

\begin{proof}[Proof of \cref{thm:open-domain-global}]
Choose \(Q\) with \(\overline\Omega\subset Q^\circ\).  The velocity--gradient
pairs of fields in \(\mathcal D_\sigma(\Omega)\), extended by zero, belong to
\(\mathcal D_\sigma(Q)\).
Consequently \(V_\Omega\) and \(H_\Omega\) are closed subspaces of
\(V_Q\) and \(H_Q\).  Restricting \(J_Q\) to \(V_\Omega\) and
regarding its range in \(H_\Omega\) gives a compact map
\(J_\Omega:V_\Omega\to H_\Omega\).  Its range is dense by construction,
and injectivity follows from the graph-gradient closure.  Compact support
gives transport integration by parts on the core.  Continuity then yields
\(b_\Omega(v,w,w)=0\) for \(v,w\in V_\Omega\).  The estimate
\cref{eq:box-ladyzhenskaya} applies to the zero extensions and bounds
\(b_\Omega\) on \(V_\Omega^3\).  The compact spectral construction
therefore gives finite-dimensional spaces, orthogonal velocity projections
\(P_m\), and compatible energy lifts on \(\Omega\).

For each integer \(N\geq0\), solve the unforced coefficient equation on
\([0,N]\).  Skew cancellation gives the exact finite-dimensional energy
identity
and bounds the velocity in \(L^\infty(0,N;H_\Omega)\) and the energy lift in
\(L^2(0,N;V_\Omega)\), uniformly in \(m\).  The Ladyzhenskaya estimate
also bounds the derivative in \(V_\Omega'\), uniformly in \(m\).  The coefficient
vector field is locally Lipschitz on each finite-dimensional space.  Two
solutions with the same initial projection on different horizons agree on
their common interval by ODE uniqueness and the energy bound.  Thus the
solutions for fixed \(m\) are consistent under restriction to shorter
intervals.

On each \([0,N]\), the spectral compactness theorem gives strong
\(L^2(0,N;H_\Omega)\) precompactness.  A diagonal extraction over the
integer horizons produces one subsequence converging strongly on every
horizon.  The uniform energy bounds and weak sequential compactness of the
separable Hilbert spaces \(L^2(0,N;V_\Omega)\) refine it to a subsequence
converging weakly in the energy space on every horizon.  A further diagonal
extraction makes every fixed finite spectral projection converge uniformly
in time on every \([0,N]\); the strong and weak limits are unchanged.
The finite-horizon limit argument then gives \(u_V^N\), a weakly continuous
velocity representative \(u_c^N\), the weak equation, and the energy
inequality for every time in \([0,N]\).

It remains to identify the representatives on overlaps.  If \(N\leq M\),
horizon compatibility of the Galerkin paths and uniqueness of uniform
limits imply
\[
  \lim_{k\to\infty}P_j u_{m_k}^{N}(t)
  =\lim_{k\to\infty}P_j u_{m_k}^{M}(t)
  \qquad (t\in[0,N],\ j\in\mathbb N).
\]
Self-adjointness of \(P_j\) identifies these limits with the pairings of
\(u_c^N(t)\) and \(u_c^M(t)\) against \(P_jy\), for every
\(y\in H_\Omega\).  Since \(P_jy\to y\) in \(H_\Omega\), the two
representatives agree at every \(t\in[0,N]\).  Define \(u_c(t)\) using
any integer horizon containing \(t\).  Its restriction to each
\([0,N]\) is \(u_c^N\), so it is weakly continuous, attains \(u_0\),
and has the asserted equation and energy inequality on every finite
horizon.
\end{proof}

\section{The Forced Problem}\label{sec:forced}

\Cref{thm:variational-finite-dimensional} extends a local Galerkin solution to
the whole interval once the forcing work has an integrable bound independent
of the solution.  This bound is zero in the unforced case.  For a general
\(V_Q'\)-valued forcing, continuation follows from the absorption estimate
below.  The elementary bound
\(2\mathcal F(v)\leq f^2+\norm{v}_{V_Q}^2\) contains the full graph norm.
For \(Q\), \cref{eq:box-forcing-absorption} uses domination of the
graph norm by the gradient seminorm and absorbs the forcing work into
diffusion with a remainder independent of the solution.

For an admissible forcing \(\mathcal F\), the continuous representative \(r\)
identifies accumulated work with a bounded functional on \(L^2(I;V_Q)\).
This permits passage to the limit in the forcing term under weak convergence
in \(L^2(I;V_Q)\).

\begin{theorem}[Forced spectral Galerkin solutions]
\label{thm:forced-galerkin}
Let \(Q\subset\mathbb R^2\) be a rectangle, let \(I=[a,b]\), let
\(u_0\in H_Q\), and let \(\mathcal F\) be an admissible forcing with control function
\(f\).  Put
\[
  R=\Bigl(\norm{u_0}_{H_Q}^2+(1+C_Q^2)\int_If(t)^2\,dt\Bigr)^{1/2}.
\]
Then every Galerkin space \(H_m\) of
\cref{prop:box-spectral-galerkin-spaces} admits a solution
\(U_m:I\to H_m\) of the forced variational problem with \(U_m(a)=P_mu_0\) and
\begin{equation}\label{eq:forced-bounds}
  \sup_{t\in I}\norm{U_m(t)}_{H_m}\leq R,
  \qquad
  \int_I A_Q(E_mU_m,E_mU_m)\,dt\leq R^2,
  \qquad
  \norm{E_mU_m}_{L^2(I;V_Q)}\leq\sqrt{1+\lvert I\rvert}\,R .
\end{equation}
Moreover, there exist indices \(m_k\to\infty\) and
limits \(u_H\in L^2(I;H_Q)\) and \(u_V\in L^2(I;V_Q)\) such that
\[
  J_QE_{m_k}U_{m_k}\longrightarrow u_H\quad\text{in }L^2(I;H_Q),
  \qquad
  E_{m_k}U_{m_k}\rightharpoonup u_V\quad\text{in }L^2(I;V_Q),
\]
with \(J_Qu_V=u_H\), and every fixed finite-dimensional projection of the
\(H_Q\)-valued solutions converges uniformly on \(I\).
\end{theorem}

\begin{proof}
Apply \cref{eq:box-forcing-absorption} to \(v=E_mx\) with the value
\(f(t)\).  The diffusion form on \(H_m\) equals
\(\norm{G_QE_mx}_{L^2}^2\), so the work hypothesis of
\cref{thm:variational-finite-dimensional} holds with
\(g(t)=(1+C_Q^2)f(t)^2\), which is independent of the solution and
continuous and therefore integrable on \(I\).  The initial datum satisfies
\(\norm{P_mu_0}\leq\norm{u_0}_{H_Q}\), so
\(\norm{P_mu_0}^2+\int_Ig\leq R^2\).  Continuity of \(\mathcal F\) implies
that its restriction to \(H_m\) is bounded in the dual norm on the compact
interval, which gives a local existence interval for all initial data in the
relevant \(H_m\)-ball.
The energy estimate extends \(U_m\) to all of \(I\) and gives the first two
bounds in \cref{eq:forced-bounds}; the exact graph identity gives the third.

With \(H_*=R\), \(R_*=\sqrt{1+\lvert I\rvert}\,R\), and
\(F_*^2=\int_If^2\), \cref{prop:box-spectral-dual-l2} gives a uniform bound
for the derivative in \(L^2(I;V_Q')\).  The coordinate identity
\cref{eq:box-spectral-coordinate} converts that bound into a uniform
\(1/2\)-H\"older estimate for each fixed projection, and
\cref{thm:spectral-l2-compactness} gives a strongly convergent subsequence in
\(L^2(I;H_Q)\).  After passing to a further subsequence, Hilbert weak
compactness of the uniformly bounded lifted solutions gives the weak
\(L^2(I;V_Q)\) limit, and continuity of the operator induced by \(J_Q\) on
the Bochner spaces identifies \(J_Qu_V=u_H\).  A
diagonal subsequence gives uniform convergence of every fixed finite
spectral projection.
\end{proof}

\begin{proof}[Proof of \cref{thm:forced-leray-hopf}]
Let \(R\) be the radius in the theorem.  Take the subsequence, the limits
\(u_H\) and \(u_V\), and the uniform bounds
from \cref{thm:forced-galerkin}, and write \(u_k=E_{m_k}U_{m_k}\).  The
weakly continuous representative \(u_c\), its identification with \(u_H\)
almost everywhere, the initial value \(u_c(a)=u_0\), and the gradient bound
follow from the construction in the proof of
\cref{thm:box-energy-weak-limit}, with \(\norm{u_0}_{H_Q}\) replaced by
\(R\), using the uniform projection convergence supplied by
\cref{thm:forced-galerkin}.

For \cref{eq:forced-weak-equation}, test the finite variational system
against \(P^V_{m_k}\phi\) and a scalar time test and integrate by parts in
time.  The resulting finite identity is the unforced one with the extra term
\(\int_I\mathcal F(t)\bigl(P^V_{m_k}\phi\bigr)\eta(t)\,dt\) on the right.
Its limit is the corresponding integral for \(\phi\), because
\[
  \Bigl\lvert
    \int_I\mathcal F(t)\bigl(P^V_{m_k}\phi-\phi\bigr)\eta(t)\,dt
  \Bigr\rvert
  \leq\norm{P^V_{m_k}\phi-\phi}_{V_Q}\int_If(t)\lvert\eta(t)\rvert\,dt,
\]
the control function \(f\) is continuous and hence integrable on the compact
interval, and the contractive energy projections converge strongly to the
identity.  The time-derivative, diffusion, convection, and initial-pairing
terms converge by the estimates used for the unforced limit: the \(H_Q\)-valued solutions
converge strongly in \(L^2(I;H_Q)\), the lifted solutions converge weakly in
\(L^2(I;V_Q)\), the Ladyzhenskaya estimate and the \(H_Q\)-convergence give
strong convergence in \(L^2(I;L^4)\), so that skew-symmetry and H\"older's
inequality pass the
convection integral, and \(P_{m_k}u_0\to u_0\) in \(H_Q\).  Passing to the
limit gives \cref{eq:forced-weak-equation}.

For \cref{eq:forced-energy-inequality}, fix \(t\in I\).  Because
\(\mathcal F\) is represented by the continuous path \(r\), the accumulated
work is the value at \(u_V\) of the bounded linear functional
\[
  W_t:L^2(I;V_Q)\longrightarrow\mathbb R,
  \qquad
  W_t(w)=\bigl(\mathbf 1_{[a,t]}r,\,w\bigr)_{L^2(I;V_Q)}
  =\int_a^t\bigl(r(s),w(s)\bigr)_{V_Q}\,ds ,
\]
whose representing vector lies in \(L^2(I;V_Q)\) because \(r\) is bounded on
the compact interval.  The exact finite energy identity of
\cref{thm:variational-finite-dimensional}, expressed in \(H_Q\) and \(V_Q\),
reads
\[
  \norm{X_k(t)}^2
  =\norm{P_{m_k}u_0}_{H_Q}^2+2W_t(u_k),
  \qquad
  X_k(t)=\bigl(J_Qu_k(t),\sqrt2\,R_tG_Qu_k\bigr),
\]
in the Hilbert product
\[
  H_Q\oplus_2L^2\bigl([a,t];L^2(Q;\mathbb R^{2\times2})\bigr).
\]
The sequence \(X_k(t)\) converges weakly to
\(X(t)=\bigl(u_c(t),\sqrt2\,R_tG_Qu_V\bigr)\): the first coordinate by the
pointwise weak extraction defining \(u_c\), the second because \(R_tG_Q\) is
bounded and \(u_k\rightharpoonup u_V\) in \(L^2(I;V_Q)\).  The right-hand
side converges as well, to \(\norm{u_0}_{H_Q}^2+2W_t(u_V)\), by strong
convergence of the initial projections and weak convergence of \(u_k\)
against the fixed functional \(W_t\).  Weak lower semicontinuity therefore
gives
\[
  \norm{X(t)}^2
  \leq\liminf_{k\to\infty}\norm{X_k(t)}^2
  =\norm{u_0}_{H_Q}^2+2W_t(u_V),
\]
which is \cref{eq:forced-energy-inequality} once \(\norm{X(t)}^2\) is
expanded.  Since \(t\) was arbitrary, the inequality holds at every time.
\end{proof}

\section{Acknowledgments}
This work was partially supported by the Simons Foundation Travel Grant for
Mathematicians (No.~0007730).  The author used ChatGPT (GPT-5.4, GPT-5.5,
and GPT-5.6 Sol) and Claude in
preparing and checking the Lean formalization.

\medskip
\noindent\textbf{Code availability.}\par
\noindent
The Lean sources are available at
\url{https://github.com/wwang-math/FluidGalerkinLean}.  They use Lean
\texttt{v4.30.0-rc1} and mathlib revision
\href{https://github.com/leanprover-community/mathlib4/tree/8d6f23e07b24c7dda53bb66ba1acaf7b99c9adf6}{\texttt{8d6f23e07b24}}.

\bibliographystyle{plain}
\bibliography{references}

\end{document}